\documentclass[a4paper,11pt]{article}
\usepackage[latin1]{inputenc}
\usepackage{amsfonts}
\usepackage{amsthm}
\usepackage{amsmath}
\usepackage{amsfonts}
\usepackage{latexsym}
\usepackage{amssymb}
\usepackage{dsfont}

\newtheorem{fed}{Definition}[section]
\newtheorem*{fed*}{Definition}
\newtheorem*{feds*}{Definitions}
\newtheorem{teo}[fed]{Theorem}
\newtheorem*{teo*}{Theorem}
\newtheorem{lem}[fed]{Lemma}
\newtheorem{cor}[fed]{Corollary}
\newtheorem{pro}[fed]{Proposition}
\theoremstyle{definition}
\newtheorem{rem}[fed]{Remark}

\newtheorem*{rems*}{Remarks}

\newtheorem{exa}[fed]{Example}

\newtheorem{notas}[fed]{Notations}

\def\ga{\gamma}
\def\coma{\, , \, }

\def\py{\peso{and}}
\newcommand{\peso}[1]{ \quad \text{ #1 } \quad }

\def\n0{n_{ \text{\rm \tiny o}}}

\def\In{\mathbb {I} _n}
\def\IM{\mathbb {I} _m}
\def\M{\mathbb {M}}
\def\suml{\sum\limits}

\def\QEDP{\tag*{\QED}}

\def\bce{\begin{center}}
\def\ece{\end{center}}

\DeclareMathOperator{\FP}{FP\,}

\def\py{\peso{and}}
\def\rk{\text{\rm rk}}
\def\noi{\noindent}
\def\cF{\mathcal F}
\def\cG{\mathcal G}
\def\QED{\hfill $\square$}
\def\EOE{\hfill $\triangle$}
\def\EOEP{\tag*{\EOE}}
\def\bm{\left[\begin{array}}
\def\em{\end{array}\right]}
\def\ben{\begin{enumerate}}
\def\een{\end{enumerate}}
\def\bit{\begin{itemize}}
\def\eit{\end{itemize}}
\def\barr{\begin{array}}
\def\earr{\end{array}}
\def\igdef{\ \stackrel{\mbox{\tiny{def}}}{=}\ }

\def\la{\lambda}
\def\al{\alpha}
\def\N{\mathbb{N}}
\def\R{\mathbb{R}}\def\C{\mathbb{C}}
\def\I{\mathbb{I}}
\def\Z{\mathbb{Z}}
\def\cA{\mathcal{A}}

\def\T{\mathbb{T}}

\def\cH{\mathcal{H}}

\def\cS{{\cal S}}

\def\cM{{\cal M}}

\def\cV{{\cal V}}
\def\cU{{\cal U}}
\def\cW{{\cal W}}

\def\cX{\mathcal{X}}

\def\inc{\subseteq}

\def\api{\langle}
\def\cpi{\rangle}

\def\da{^\downarrow}

 \DeclareMathOperator{\tr}{tr}
\DeclareMathOperator{\gen}{span}

\DeclareMathOperator{\convf}{Conv (\R_{+})}
\DeclareMathOperator{\convfs}{Conv_s (\R_{+})}
\DeclareMathOperator*{\supes}{ess\,sup}
\DeclareMathOperator*{\infes}{ess\,inf}

\newcommand{\esssup}{{\mathrm{ess}\sup}}

\newcommand{\hil}{\mathcal{H}}

\newcommand{\mat}{\mathcal{M}_d(\mathbb{C})}

\newcommand{\matsad}{\mathcal{H}(d)}

\newcommand{\matud}{\mathcal{U}(d)}

\newcommand{\matpos}{\mat^+}

\newcommand{\matinvd}{\mathcal{G}\textit{l}\,(d)}

\def\beq{\begin{equation}}
\def\eeq{\end{equation}}

\def\pausa{\medskip\noi}

\def\Ax2{\,[ S_{E(\cF)^\#_\cV}]_x }
\def\fe{\peso{for every}}

\begin{document}
\title{\bf Frames of translates with prescribed fine structure \\ in shift invariant spaces}
\author{Mar\'\i a J. Benac $^{*}$, Pedro G. Massey $^{*}$, and Demetrio Stojanoff 
\footnote{Partially supported by CONICET 
(PIP 0435/10) and  Universidad Nacional de La PLata (UNLP 11X681) } \ 
 \footnote{ e-mail addresses: mjbenac@gmail.com , massey@mate.unlp.edu.ar , demetrio@mate.unlp.edu.ar}
\\ 
{\small Depto. de Matem\'atica, FCE-UNLP
and IAM-CONICET, Argentina  }}
\date{}
\maketitle

\begin{abstract}
For a given finitely generated shift invariant (FSI) subspace $\cW\subset L^2(\R^k)$ we obtain a simple criterion for the existence of shift generated (SG) Bessel sequences $E(\cF)$ induced by finite sequences of vectors $\cF\in \cW^n$ that have a prescribed fine structure i.e., such that the norms of the vectors in $\cF$ and the spectra of $S_{E(\cF)}$ is prescribed in each fiber of $\text{Spec}(\cW)\subset \T^k$. We complement this result by developing an analogue of the so-called sequences of eigensteps from finite frame theory in the context of SG Bessel sequences, that allows for a detailed description of all sequences with prescribed fine structure. Then, given $0<\alpha_1\leq \ldots\leq \alpha_n$ we characterize the
finite sequences $\cF\in\cW^n$ such that $\|f_i\|^2=\alpha_i$, for $1\leq i\leq n$, and such that the fine spectral structure of the shift generated Bessel sequences $E(\cF)$ have minimal spread (i.e. we show the existence of optimal SG Bessel sequences with prescribed norms); in this context the spread of the spectra is measured in terms of the convex potential $P^\cW_\varphi$ induced by $\cW$ and an arbitrary convex function $\varphi:\R_+\rightarrow \R_+$. 
\end{abstract}



\tableofcontents

\section{Introduction}
Let $\cW$ be a closed subspace of a separable complex Hilbert space $\hil$ and let $\mathbb I$ be a finite or countable infinite set. A sequence $\cF=\{f_i\}_{i\in \mathbb I}$ in $\cW$ is a frame for $\cW$ if there exist positive constants $0<a\leq b$ such that 
$$ 
a \, \|f\|^2\leq \sum_{i\in \mathbb I}|\langle f,f_i\rangle |^2\leq b\,\|f\|^2 
\peso{for every} f\in \cW\, . 
$$ 
If we can choose $a=b$ then we say that $\cF$ is a tight frame for $\cW$.
A frame $\cF$ for $\cW$ allows for linear (typically redundant) and stable encoding-decoding schemes of vectors (signals) in $\cW$. Indeed, if 
$\cV$ is a closed subspace of $\hil$ such that $\cV\oplus\cW^\perp=\hil$ (e.g. $\cV=\cW$) then  it is possible to find frames $\cG=\{g_i\}_{i\in \mathbb I}$ for $\cV$ such that 
\begin{equation}\label{eq: intro duals}
f=\sum_{i\in \mathbb I}\langle f,g_i\rangle \ f_i \ , \quad \text{ for } f\in\cW\,.  
\end{equation}
The representation above lies within the theory of oblique duality (see \cite{YEldar3,CE06,YEldar1,YEldar2}). 
In applied situations, it is usually desired to develop encoding-decoding schemes as above, 
with some additional features related with stability of the scheme. 
In some cases, we search for schemes such that the sequence of norms $\{\|f_i\|^2\}_{i\in \mathbb I}$ as well as the spectral properties of the family $\cF$ are given in advance, leading to what is known in the literature as frame design problem (see \cite{AMRS,BF,CFMPS,CasLeo,MR08,Pot} and the papers \cite{FMP,MRS13,MRS14,MRS13b} for the more general frame completions problem with prescribed norms). It is well known that both the spread of the sequences of norms as well as the spread of the spectra of the frame $\cF$ are linked with numerical properties of $\cF$.
Once we have constructed a frame $\cF$ for $\cW$ with the desired properties, we turn our attention to the construction of frames $\cG$ 
for $\cV$ satisfying Eq.\eqref{eq: intro duals} and having some prescribed features related with their numerical stability
(see \cite{BMS14,BMS15,CE06,MRS13,MRS13b}).

\pausa
It is well known that the frame design problem has an equivalent formulation in terms of the relation between the main diagonal of a positive semi-definite
operator and its spectra; in the finite dimensional setting this relation is characterized in the Schur-Horn theorem from matrix analysis. There has been recent important advances in both the frame design problems as well as the Schur-Horn theorems in infinite dimensions, mainly due to the interactions of these problems
(see \cite{AMRS,BoJa,BoJa2,BoJa3,Jas,KaWe}). There are also complete parametrizations of all finite frames with prescribed norms and eigenvalues (of their frame operators) in terms of the so-called eigensteps sequences \cite{CFMPS}.
On the other hand, the spectral structure of oblique duals (that include classical duals) of a fixed frame can be described in terms of the relations between the spectra of a positive semi-definite operator and the spectra of its compressions to subspaces. In the finite dimensional context (see \cite{BMS14,MRS13}) these relations are known as the Fan-Pall inequalities (that include the so-called interlacing inequalities as a particular case). 
Yet, in general, the corresponding results in frame theory do not take into consideration any additional structure of the frame. For example, regarding the frame design problem, it seems natural to wonder whether we can construct a structured frame (e.g., wavelet, Gabor or a shift generated frame) with prescribed structure; similarly, in case we fix a structured frame $\cF$ for $\cW$ it seems natural to wonder whether we can construct structured oblique dual frames with further prescribed properties.

\pausa
In \cite{BMS15}, as a first step towards a detailed study of the spectral properties of structured oblique duals of shift generated systems induced by finite families of vectors in $L^2(\R^k)$,  we extended the Fan-Pall theory to the context of measurable fields of positive semi-definite matrices 
and their compressions by measurable selections of subspaces; this allowed us to give an explicit description of what we called {\it fine spectral structure} of the shift generated duals of a fixed shift generated (SG) frame for a finitely generated shift invariant (FSI) subspace $\cW$ of $L^2(\R^k)$. Given a convex function $\varphi:\R_+\rightarrow \R_+$ we also introduced 
the convex potential associated to pair the $(\varphi,\cW)$, that is a functional on SG Bessel sequences that measures the spread of the fine spectral structure of the sequence; there we showed that these convex potentials detect tight frames as their minimizers (under some normalization conditions). 
Yet, our analysis was based on the fine spectral structure of a given SG Bessel sequence in a FSI subspace 
$\cW\subset L^2(\R^k)$.

\pausa
 In this paper, building  on an extension of the Schur-Horn theorem for measurable fields of positive semi-definite matrices, 
we characterize the possible {\it fine structures} of SG Bessel sequences in FSI subspaces (see Section \ref{SI cosas} for preliminaries on SG Bessel sequences, Remark \ref{rem sobre estruc fina} and Theorem \ref{teo sobre disenio de marcos}); 
 thus, we solve a frame design problem, where the prescribed features of the SG Bessel sequences are described in terms of some internal (or fine) structure, relative to a finitely generated shift invariant subspace $\cW$. We also show that the Fan-Pall theory for fields of positive semi-definite matrices can be used to obtain a detailed description of SG Bessel sequences with prescribed fine structure, similar to that obtained in terms of the eigensteps in \cite{CFMPS}. In turn, we use these results to show that 
 given a FSI subspace $\cW$, a convex function $\varphi:\R_+\rightarrow \R_+$ and a finite sequence of positive numbers $\alpha_1\geq \ldots\geq \alpha_n>0$, there exist vectors $f_i\in\cW$ such that $\|f_i\|^2=\alpha_i$, for $1\leq i\leq n$, and such that the SG Bessel sequence induced by these vectors minimizes the convex potential associated to the pair $(\varphi,\cW)$, among all such SG Bessel sequences (for other optimal design problems in shift invariant spaces see \cite{AlCHM1,AlCHM2}). The existence of these  $(\varphi,\cW)$-optimal shift generated frame designs with prescribed norms is not derived using a direct ``continuity + compactness'' argument. Actually, their existence follows from a discrete nature of their spectral structure; we make use of the this fact to reduce the problem of describing the structure of optimal designs, to an optimization problem in a finite dimensional setting. As a tool, we consider the waterfilling construction in terms of majorization in general probability spaces. It is worth pointing out that there has been interest in the structure of finite sequences of vectors that minimize convex potentials in the finite dimensional context (see \cite{CKFT,FMP,MR08,MR10}), originating from the seminal paper \cite{BF}; our present situation is more involved and, although we reduce the problem to a finite dimensional setting, this reduction is not related with the techniques nor the results of the previous works on finite families of vectors.

 \pausa
The paper is organized as follows. In Section \ref{sec prelim}, after fixing the general notations used in the paper, we present some preliminary material on frames, shift invariant subspaces and shift generated Bessel sequences; we end this section with the general notion of majorization in probability spaces.
In Section \ref{subsec exac charac} we obtain an exact characterization of the existence of shift generated Bessel sequences with prescribed fine structure in terms of majorization relations; this result is based on a version of the Schur-Horn theorem for measurable fields of positive semi-definite matrices (defined on measure spaces) that is developed in the appendix (see Section \ref{Appendixity}). In Section \ref{subsec eigensteps}, building on the Fan-Pall inequalities from \cite{BMS15}, we obtain a detailed description of all shift generated Bessel sequences with prescribed fine structure that generalizes the so-called eigensteps construction in the finite dimensional setting. In Section 
\ref{sec opti frames with prescribed norms} we show that for a fixed sequence of positive numbers $\alpha_1\geq \ldots\geq \alpha_n>0$, a 
convex function $\varphi:\R_+\rightarrow \R_+$ and a FSI subspace $\cW\subset L^2(\R^k)$ there exist vectors $f_i\in\cW$ such that $\|f_i\|^2=\alpha_i$, for $1\leq i\leq n$, and such that $\cF$ minimizes the convex potential associated to the pair $(\varphi,\cW)$ among all such finite sequences; in order to do this, we first consider in Section \ref{subse uniform} the uniform case in which the dimensions of the fibers of $\cW$ are constant on the spectrum of $\cW$. 
The general case of the optimal design problem with prescribed norms in a FSI is studied in Section \ref{subsec gral mi gral}; our approach is based on a reduction of the problem to an optimization procedure in the finite dimensional setting. The paper ends with an Appendix, in which we consider a measurable version of the Schur-Horn theorem needed in Section \ref{subsec exac charac} as well as some technical aspects of an optimization problem needed in Section \ref{subsec gral mi gral}.

\section{Preliminaries} \label{sec prelim}

In this section we recall some basic facts related with frames for subspaces and shift generated frames for shift invariant (SI) subspaces of $L^2(\R^k)$. 
At the end of this section we describe majorization between functions in arbitrary probability spaces.

\def\IM{\mathbb {I} _m}
\def\ID{\mathbb {I} _d}
\def\uno{\mathds{1}}

\pausa
{\bf General Notations} 

\pausa
Throughout this work we shall use the following notation:
the space of complex $d\times d$ matrices is denoted by $\mat$, the real subspace of self-adjoint matrices is denoted $\matsad$ and $\matpos$ denotes the set of positive semi-definite matrices; $\matinvd$ is the group of invertible elements of $\mat$, $\matud$ is the subgroup of unitary matrices and $\matinvd^+ = \matpos \cap \matinvd$. If $T\in \mat$, we  denote by   
$\|T\|$ its spectral norm,
by $\rk\, T= \dim R(T) $  the rank of $T$,
and by $\tr T$ the trace of $T$. 

\pausa
Given $d \in \N$ we denote by $\ID = \{1, \dots , d\} \inc \N$ and we set $\I_0=\emptyset$.
For a vector $x \in \R^d$ we denote  
by $x^\downarrow\in \R^d$ 
 the rearrangement
of $x$ in  non-increasing order. We denote by
$(\R^d)^\downarrow = \{ x\in \R^d : x = x^\downarrow\}$ the set of downwards ordered vectors.
Given  $S\in \matsad$, we write $\la(S) = \la\da(S)= (\la_1(S) \coma \dots \coma \la_d(S)\,) \in 
(\R^d)^\downarrow$ for the 
vector of eigenvalues of $S$ - counting multiplicities - arranged in decreasing order.

\pausa
If $W\inc \C^d$ is a subspace we denote by $P_W \in \matpos$ the orthogonal 
projection onto $W$. 
Given $x\coma y \in \C^d$ we denote by $x\otimes y \in \mat$ the rank one 
matrix given by 
\beq \label{tensores}
x\otimes y \, (z) = \api z\coma y\cpi \, x \peso{for every} z\in \C^d \ .
\eeq
Note that, if $x\neq 0$, then
the projection $P_x \igdef P_{\gen\{x\}}= \|x\|^{-2} \, x\otimes x \,$.

\subsection{Frames for subspaces}\label{sec defi frames subespacios}

In what follows $\hil$ denotes a separable complex Hilbert space and $\mathbb I$ denotes a finite or countable infinite set. 
Let $\cW$ be a closed subspace of $\hil$: recall that a sequence $\cF=\{f_i\}_{i\in \mathbb I}$ in $\cW$ is a {\it frame} for $\cW$ 
if there exist positive constants $0<a\leq b$ such that 
\beq\label{defi frame} 
a \, \|f\|^2\leq \sum_{i\in \mathbb I}|\langle f,f_i\rangle |^2\leq b\,\|f\|^2 
\peso{for every} f\in \cW\, . 
\eeq
In general, if $\cF$ satisfies the inequality to the right in Eq. \eqref{defi frame}
we say that $\cF$ is a $b$-Bessel sequence for $\cW$. Moreover, we shall say that a sequence
$\cG=\{g_i\}_{i\in\mathbb I}$ in $\hil$ is a Bessel sequence - without explicit reference to a closed subspace - whenever
 $\cG$ is a Bessel sequence for its closed linear span; notice that this is equivalent to the fact that $\cG$ is a Bessel sequence for $\hil$.

\pausa 
Given a Bessel sequence $\cF=\{f_i\}_{i\in \mathbb I}$ we consider its {\it synthesis operator} $T_\cF\in L(\ell^2(\mathbb I),\hil)$ given by $T_\cF((a_i)_{i\in \mathbb I})=\sum_{i\in \mathbb I} a_i\ f_i$ which, by hypothesis on $\cF$, is a bounded linear transformation. We also consider $T_\cF^*\in L(\hil,\ell^2(\mathbb I))$ called the {\it analysis operator} of $\cF$, given by $T_\cF^*(f)=(\langle f,f_i\rangle )_{i\in \mathbb I}$ and the {\it frame operator} of $\cF$ defined by $S_\cF=T_\cF\,T_\cF^*$. It is straightforward to check that
$$ 
\langle S_\cF f,f \rangle =\sum_{i\in \mathbb I}|\langle f,f_i\rangle |^2 \fe
 f\in \hil\ . 
 $$
Hence, $S_\cF$ is a positive semi-definite bounded operator; moreover, a Bessel sequence $\cF$ in $\cW$ 
is a frame for $\cW$ if and only if $S_\cF$ is an invertible operator when restricted to $\cW$ or equivalently, if the range of $T_\cF$ coincides with $\cW$.

\pausa
If $\cV$ is a closed subspace of $\hil$ such that $\cV\oplus\cW^\perp=\hil$ (e.g. $\cV=\cW$) then  it is possible to find frames $\cG=\{g_i\}_{i\in \mathbb I}$ for $\cV$ such that 
$$f=\sum_{i\in \mathbb I}\langle f,g_i\rangle \ f_i \ , \quad \text{ for } f\in\cW\,.  $$
The representation above lies within the theory of oblique duality (see \cite{YEldar3,CE06,YEldar1,YEldar2}). 
In this note we shall not be concerned with oblique duals; nevertheless, notice that the numerical stability of the encoding-decoding scheme above depends both on the numerical stability corresponding to $\cF$ and $\cG$ as above. One way to measure stability of the encoding or decoding algorithms is to measure the spread of the spectra of the frame operators corresponding to $\cF$ and $\cG$. Therefore both the task of constructing optimally stable $\cF$ together with obtaining 
optimally stable duals $\cG$ of $\cF$ are of fundamental interest in frame theory.

\subsection{SI subspaces, frames of translates and their convex potentials}\label{SI cosas}

In what follows we consider $L^2(\R^k)$ (with respect to Lebesgue measure) as a separable and complex Hilbert space.
Recall that a closed subspace $\cV\subseteq L^2(\R^k)$ is {\it shift-invariant} (SI) if $f\in \cV$
implies $T_\ell f \in \cV$ for any $\ell\in \Z^k$, where $T_yf(x)=f(x-y)$ is
the translation by $y \in \R^k$. 
For example, take a subset $\cA \subset
L^2(\R^k)$ and set
$$
\cS(\cA)= \overline{\text{span}}\,\{T_\ell f:\ f\in\cA\, , \ \ell\in\mathbb Z^k\}  \,.
$$
Then, $\cS(\cA)$ is a shift-invariant subspace called the {\it SI subspace generated by $\cA$}; indeed, $\cS(\cA)$ is the smallest SI subspace that contains $\cA$. We say that a SI subspace $\cV$ is {\it finitely generated} (FSI) if there exists a finite set $\cA\subset L^2(\R^k)$ such that $\cV=\cS(\cA)$. We further say that $\cW$ is a principal SI subspace if there exists $f\in L^2(\R^k)$ such that $\cW=\cS(f)$.

\pausa In order to describe the fine structure of a SI subspace we consider the following representation of $L^2(\R^k)$ (see \cite{BDR,Bo,RS95} and \cite{CabPat} for extensions of these notions to the more general context of actions of locally compact abelian groups). Let $\T=[-1/2,1/2)$ endowed with the Lebesgue measure and let $L^2(\T^k, \ell^2(\Z^k))$ be the Hilbert space of square integrable $\ell^2(\Z^k)$-valued functions that consists of all vector valued measurable functions $\phi: \T^k \to \ell^2(\Z^k)$ with the norm 
$$\| \phi\|^2= \int_{\T^k} \| \phi(x)\|_{\ell^2(\Z^k)}^{2} \ dx< \infty.$$
Then, $\Gamma: L^2(\R^k)\to L^2(\T^k, \ell^2(\Z^k))$ defined for $f\in L^1(\R^k)\cap L^2(\R^k)$ by 
\beq\label{def: iso}
\Gamma f: \T^k \to \ell^2(\Z^k)\ ,\quad\Gamma f(x)= (\hat{f}(x+\ell))_{\ell\in \Z^k},
\eeq
extends uniquely to an isometric isomorphism between $L^2(\R^k)$ and $L^2(\T^k, \ell^2(\Z^k))$; here
$$\hat f(x)=  \int_{\R^k} f(y) \ e^{-2\pi\, i\,\langle y,\,x\rangle} \ dy \quad \text{ for } \quad x\in\R^k\, ,  $$ denotes the Fourier transform of $f\in L^1(\R^k)\cap L^2(\R^k)$.

\pausa 
Let $\cV\subset L^2(\R^k)$ be a SI subspace. Then, there exists a function $J_\cV:\T^k\rightarrow\{$ closed subspaces of 
$\ell^2(\Z^k)\}$ such that: if $ P_{J_\cV(x)}$ denotes the orthogonal projection onto $J_\cV(x)$ for $x\in\T^k$, then for every $\xi,\,\eta\in \ell^2(\Z^k)$ the function $x\mapsto \langle P_{J_\cV(x)} \,\xi\coma \eta\rangle$ is measurable and 
\beq\label{pro: V y J}
\cV=\{ f\in L^2(\R^k): \Gamma f(x) \in J_\cV(x) \,\ \text{for a.e.}\,\ x\in \T^k\}.
\eeq The function $J_\cV$ is the so-called {\it measurable range function} associated with $\cV$. By \cite[Prop.1.5]{Bo}, Eq. \eqref{pro: V y J} establishes a bijection between 
SI subspaces of $L^2(\R^k)$ and measurable range functions.
In case 
$\cV=S(\mathcal A) \subseteq L^2(\R^k)$ is the SI subspace generated by $\mathcal A=\{h_i:i\in \mathbb I\}\subset L^2(\R^k)$, where $\mathbb I$ is a finite or countable infinite set, then for a.e. $x\in\T^k$ we have that 
\beq\label{eq Jv}
J_\cV(x)=\overline{\text{span}}\,\{\Gamma h_i(x): \ i\in \mathbb I\}\,.
\eeq

\pausa
Recall that a bounded linear operator 
$S\in L(L^2(\R^k))$ is {\it shift preserving} (SP) if 
$T_\ell \, S=S\,T_\ell$ for every $\ell\in\Z^k$. In this case (see \cite[Thm 4.5]{Bo}) there exists
a (weakly) measurable field of operators $[S]_{(\cdot)}:\T^k\rightarrow \ell^2(\Z^k)$ (i.e. such that 
for every $\xi,\,\eta\in \ell^2(\Z^k)$ the function 
$\T^k\ni x\mapsto \langle [S]_x\, \xi\coma \eta \rangle $ 
is measurable) and essentially bounded (i.e. the function $\T^k\ni x\mapsto \|\,[S]_{x}\,\|$ is essentially bounded)
such that 
\beq\label{defi hatS}
 [S]_x \big(\Gamma f(x)\,\big)=\Gamma (Sf) (x) \quad \text{ for a.e. }x\in\T^k\ , \ \ f\in L^2(\R^k)\,.
\eeq Moreover, $\|S\|=\esssup_{x\in\T^k} \|\, [S]_x\, \|$.
Conversely, if $s:\T^k\rightarrow L(\ell^2(\Z^k))$ is a weakly measurable and essentially bounded field of operators then,
 there exists a unique bounded operator $S\in L(L^2(\R^k))$ that is SP and such that $[S]=s$. 
For example, let $\cV$ be a SI subspace and consider $P_\cV\in L(L^2(\R^k))$, the orthogonal projection 
onto $\cV$; then, $P_\cV$ is SP so that 
$[P_\cV]{} :\T^k\rightarrow L(\ell^2(\Z^k))$ is given by $[P_\cV]{}_x=P_{J_\cV(x)}$ i.e., the orthogonal projection onto $J_\cV(x)$, for a.e. $x\in\T^k$.

\pausa The previous notions associated with SI subspaces and SP operators allow to develop a detailed study of frames of translates. 
Indeed, let $\cF=\{f_i\}_{i\in \mathbb I}$ be a (possibly finite) sequence in $L^2(\R^k)$. In what follows we consider the sequence of integer translates
of $\cF$, denoted $E(\cF)$ and given by
$$E(\cF)=\{T_\ell \, f_i\}_{(\ell,\, i)\,\in\, \Z^{k}\times \mathbb I}\,.$$
 For $x\in \T^k$, let $\Gamma\cF(x)=\{\Gamma f_i(x)\}_{i\in \mathbb I}$ which is a (possibly finite) 
sequence in $\ell^2(\Z^k)$. Then  $E(\cF)$ is a $b$-Bessel sequence if and only if $\Gamma\cF(x)$ is a $b$-Bessel 
sequence for a.e. $x\in \T^k$ (see \cite{Bo,RS95}). In this case,  
we consider the  synthesis operator $T_{\Gamma\cF(x)}:\ell^2(\mathbb I)\rightarrow \ell^2(\Z^k)$ and frame operator $S_{\Gamma\cF(x)}:\ell^2(\Z^k)\rightarrow \ell^2(\Z^k)$  of $\Gamma\cF(x)$,  for $x\in\T^k$. It is straightforward to check that $S_{E(\cF)}$ is a SP operator.

\pausa
If $\cF=\{f_i\}_{i\in \mathbb I}$ and $\cG=\{g_i\}_{i\in \mathbb I}$ are such that $E(\cF)$ and $E(\cG)$ are Bessel sequences then (see \cite{HG07,RS95}) the following fundamental relation holds: 
\beq\label{eq:fourier}
[T_{E(\cG)}\,T^*_{E(\cF)}]_x 
= T_{\Gamma\cG(x)}\,T^*_{\Gamma\cF(x)}\ , \quad  \text{for a.e }\, x \in \T^k \,.
\eeq
These equalities 
have several consequences. For example,  if 
$\cW$ is a SI subspace of $L^2(\R^k)$ and we assume further that $\cF,\,\cG\in\cW^n$ then, 
for every $f,\,g\in L^2(\R^k)$, 
$$
\langle S_{E(\cF)}\, f,\,g\rangle =\int_{\T^k} \langle S_{\Gamma\cF(x)} 
\ \Gamma f(x),\,\Gamma g(x)\rangle_{\ell^2(\Z^k)}\ dx\ . 
$$ 
This last fact implies that $[S_{E(\cF)}]_x=S_{\Gamma \cF(x)}$ for a.e. $x\in\T^k$. Moreover, 
$E(\cF)$ is a frame for $\cW$ with frame bounds 
$0<a\leq b$ if and only if $\Gamma\cF(x)$ is a 
frame for $J_\cW(x)$ with frame bounds $0<a\leq b$ for a.e. $x\in \T^k$ (see \cite{Bo}).

\pausa
We end this section with the notion of convex potentials in FSI introduced in \cite{BMS15}\,; in order to describe these potentials we 
consider the sets 
\beq\label{def convf}
\convf = \{ 
\varphi:\R_+\rightarrow \R_+\ , \  \varphi   \ \mbox{ is a convex function}  \ \} 
\eeq
and $\convfs = \{\varphi\in \convf \ , \  \varphi$ is strictly convex $\}$.

\begin{fed} \label{defi pot conv}\rm Let $\cW$ be a FSI subspace in $L^2(\R^k)$, let $\cF=\{f_i\}_{i\in\In}\in\cW^n$ 
be such that $E(\cF)$ is a Bessel sequence and consider $\varphi\in \convf$.
The convex potential associated to $(\varphi,\cW)$ on $E(\cF)$, denoted $P_\varphi^\cW(E(\cF))$, is given by
\beq\label{eq defi pot}
P_\varphi^\cW(E(\cF))=\int_{\T^k} 
\tr(\varphi(S_{\Gamma \cF(x)})\, [P_\cW]_x) \ dx
\eeq 
where  $\varphi(S_{\Gamma \cF(x)})$ denotes the functional calculus of the positive and finite rank operator $S_{\Gamma \cF(x)}\in L(\ell^2(\Z^k))^+$ and $\tr(\cdot)$ denotes the usual semi-finite trace in $L(\ell^2(\Z^k))$.
\EOE
\end{fed}

\begin{exa}
Let $\cW$ be a FSI subspace of $L^2(\R^k)$ and let $\cF=\{f_i\}_{i\in\In}\in\cW^n$. If we set $\varphi(x)=x^2$ for $x\in \R_+$ then, the corresponding potential on $E(\cF)$, that we shall denote $\FP(E(\cF))$, is given by
$$ 
\FP(E(\cF))= \int_{\T^k}
\tr( S_{\Gamma \cF(x)}^2) \ dx= \int_{\T^k} \ \sum_{i,\,j\in \In} |\langle \Gamma f_i(x),\Gamma f_j(x)\rangle|^2 \ dx\,, $$
where we have used the fact that $\varphi(0)=0$ in this case. Hence,   $\FP(E(\cF))$ is a natural extension of the Benedetto-Fickus frame potential (see \cite{BF}).
\EOE
\end{exa}

\pausa
With the notation of Definition \ref{defi pot conv}, it is shown in \cite{BMS15} that $P_\varphi^\cW(E(\cF))$ is a well defined functional on the class of Bessel sequences $E(\cF)$ induced by a finite sequence $\cF=\{f_i\}_{i\in\In}\in\cW^n$ as above. The main motivation for considering convex potentials is 
that, under some natural normalization hypothesis, they detect tight frames as their minimizers 
(see \cite[Theorem  3.9.]{BMS15} or Corollary \ref{cororo1} below); that is,
convex potentials provide simple scalar measures of stability that can be used to compare shift generated frames. 
Therefore, the convex potentials for FSI are natural extensions of the convex potentials in finite dimensions introduced in \cite{MR10}. 
In what follows, we shall consider the existence of tight frames $E(\cF)$ for the FSI $\cW$ with prescribed norms. 
It turns out that there are natural restrictions for the existence of such frames (see Theorem \ref{teo sobre disenio de marcos} below). In case these restrictions are not fulfilled then, the previous remarks show that minimizers of convex potentials associated to a pair $(\varphi,\,\cW)$ within the class of frames with prescribed norms are natural substitutes of tight frames.

\subsection{Majorization in probability spaces}\label{2.3}

Majorization between vectors (see \cite{Bhat,MaOl}) has played a key role in frame theory. On the one hand, majorization allows to characterize the existence of frames with prescribed properties (see \cite{AMRS,CFMPS,CasLeo}). On the other hand, majorization is a preorder relation that implies a family of tracial inequalities; this last fact can be used to explain the structure of minimizers of general convex potentials, that include the Benedetto-Fickus' frame potential (see \cite{BF,CKFT,MR08,MR10,MRS13,MRS14,MRS13b}). We will be dealing with convex potentials in the context of Bessel families of integer translates of finite sequences; accordingly, we will need the following general notion of majorization between functions in probability spaces.

\pausa
Throughout this section the triple $(X,\cX,\mu)$ denotes 
a probability space i.e. 
 $\cX$ is a $\sigma$-algebra of sets in $X$ and $\mu$ is a probability measure defined on $\cX$. We shall denote by $L^\infty(X,\mu)^+ 
= \{f\in L^\infty(X,\mu): f\ge 0\}$. 
For $f\in L^\infty(X, \mu)^+$,
the {\it decreasing rearrangement} of $f$ (see \cite{MaOl}), denoted $f^*:[0,1)\rightarrow \R_+$, is given by
\beq\label{eq:reord}
f^*(s ) \igdef\sup \,\{ t\in \R_+ : \ \mu \{x\in X:\ f(x)>t\} >s\} 
\fe s\in [0,1)\, .
\eeq 

\begin{rem} \label{rem:prop rear elem}We mention some elementary facts related with the decreasing rearrangement of functions that we shall need in the sequel. Let $f\in L^\infty(X,\mu)^+$, then: 
\ben
\item $f^*$ is a right-continuous and non-increasing function. 
\item $f$ and $f^*$ are equimeasurable i.e. for every Borel set $A\subset \R$ then $\mu(f^{-1}(A))=|(f^*)^{-1}(A)|$, where $|B|$ denotes the Lebesgue measure of the Borel set $B\subset \R$. In turn, this implies that for every continuous $\varphi:\R_+\rightarrow \R_+$ then: $\varphi\circ f\in L^\infty(X,\mu)$ iff $\varphi\circ f^*\in L^\infty([0,1])$ and in this case 
\beq\label{reor int} 
\int_X \varphi\circ f\ d\mu =\int_{0}^1 \varphi\circ f^*\ dx\ . 
\eeq
\item If $g\in L^\infty(X,\mu)$ is such that $f\leq g$ then $0\leq f^*\leq g^*$; moreover, in case $f^*=g^*$ then $f=g$. \EOE
\een
\end{rem}

\begin{fed}\rm Let $f, g\in L^\infty (X, \mu)^+$ and 
let $f^*,\,g^*$ denote their decreasing rearrangements. We say that $f$ \textit{submajorizes} $g$ (in $(X,\cX,\mu)$), denoted $g \prec_w f$, if
\begin{eqnarray*}\label{eq: mayo de func}
\int_{0}^{s} g^*(t) \,\ dt &\leq& \int_{0}^{s} f^*(t)\,\ dt \fe  0\leq s\leq 1 \,.
\end{eqnarray*}
If in addition $\int_{0}^{1} g^*(t)\,\ dt = \int_{0}^{1} f^*(t)\,\ dt$ 
 we say that $f$ \textit{majorizes} $g$ and write $g \prec f$. \EOE
\end{fed}
\pausa
In order to check that majorization holds between functions in probability spaces, we can consider the so-called {\it doubly stochastic maps}. Recall that a linear operator $D$ acting on $L^\infty(X,\mu)$ is a doubly-stochastic map if $D$ is unital, positive and trace preserving i.e. 
$$ 
D(1_X)=1_X \ , \ \ 
D\big(\, L^\infty (X, \mu)^+ \, \big)\inc L^\infty (X, \mu)^+ 
\py 
\int_X D(f)(x)\ d\mu(x) =\int_X f(x)\ d\mu(x)  
$$ 
for every $f\in L^\infty(X,\mu)$. 
It is worth pointing out that $D$ is necessarily a contractive map.

\pausa
Our interest in majorization relies in its relation with integral inequalities in terms of convex functions. The following result summarizes this relation as well as the role of the doubly stochastic maps (see for example \cite{Chong,Ryff}). Recall that $\convf$ and $\convfs$ (see Eq. \eqref{def convf}) denote the sets of convex and strictly convex functions $\varphi:\R_+\rightarrow \R_+$, respectively.

\begin{teo}\label{teo porque mayo} \rm
Let $f,\,g\in  L^\infty (X, \mu)^+$. Then the following conditions are equivalent:
\ben
\item $g\prec f$;
\item There is a doubly stochastic map $D$ acting on $L^\infty(X,\mu)$ such that $D(f)=g$;
\item For every $\varphi \in \convf$ we have that 
\beq\label{eq teo:desi mayo}
 \int_X \varphi(g(x)) \ d\mu(x)\leq  \int_X \varphi(f(x))\ d\mu(x)\ .
\eeq
\een 
Similarly,  $g\prec_w f \iff $  Eq. \eqref{eq teo:desi mayo} holds for every {non-decreasing} convex function $\varphi$. \qed
\end{teo}

\pausa
The following result plays a key role in the study of the structure of minimizers of $\prec_w$ within (appropriate) sets of functions.

\begin{pro}[\cite{Chong}]\label{pro int y reo} \rm
Let $f,\,g\in L^\infty(X,\mu)^+$ such that $g\prec_w f$. If there 
exists $\varphi\in\convfs$ such that 
\beq 
\int_X \varphi(f(x))\ d\mu(x) =\int_X \varphi(g(x)) \ d\mu(x) 
\peso{then} g^*=f^* \ . \QEDP
\eeq 
\end{pro}
\pausa

\section{Existence of shift generated frames with prescribed fine struture}

In this section we characterize the fine structure of a Bessel sequence $E(\cF)$, where $\cF=\{f_i\}_{i\in\In}\in\cW^n$. By the fine (or relative) structure of $E(\cF)$ we mean the sequence of norms of the vectors $\Gamma \cF(x)=(\Gamma f_i(x)) _{i\in\In}$ and the sequence of eigenvalues of $[S_{E(\cF)}]_x$ for $x\in\T^k$ (see Remark \ref{rem sobre estruc fina} for a precise description). As we shall see, the possible fine structure of $E(\cF)$ can be described in terms of majorization relations. 

\subsection{A complete characterization in terms of majorization relations}\label{subsec exac charac}

\pausa
We begin by showing the existence of measurable 
spectral representations of self-adjoint SP operators with range lying in a FSI subspace (see Lemma \ref{lem spect represent ese}), 
which follow from results from \cite{RS95} regarding the existence of measurable fields of eigenvectors and eigenvalues (counting multiplicities and arranged in non-increasing order) of measurable fields $M:\T^k \rightarrow \matsad$ of selfadjoint matrices.
In order to do that, we first recall some notions and results from \cite{Bo}. 

\pausa
Given $\cW\subset L^2(\R^k)$ a FSI subspace, we say that $f\in \cW$ is a quasi-orthogonal generator of $\cW$ if
\beq\label{eq:}
\|g\|^2=\sum_{\ell\in \Z^k} |\langle T_\ell f, g \rangle|^2\ , \peso{for every} g\in \cW\, .
\eeq
The next theorem, which is a consequence of results from \cite{Bo}, provides a decomposition of any FSI subspace 
of $L^2(\R^n)$ into a finite orthogonal sum of principal SI subspaces with quasi-orthogonal generators.

\begin{teo}[\cite{Bo}]\label{teo:la descom de Bo} \rm
Let $\cW$ be a FSI subspace of $L^2(\R^k)$, with $d=\supes_{x\in \T^k} d(x)$, where $d(x)=\dim J_{\cW}(x)$ for $x\in\T^k$. Then there exist $h_1,\ldots,h_d\in\cW$ such that $\cW$ can be decomposed as an orthogonal sum
\begin{eqnarray}
\cW=\bigoplus_{j\in\I_{d}} \cS(h_j),
\end{eqnarray}
where $h_j$ is a quasi orthogonal generator of $\cS(h_j)$ for $j\in\I_{d}\,$, and $\text{Spec}(\cS(h_{j+1}))\subseteq 
\text{Spec}(\cS(h_j))$ for $j\in \I_{d-1}\,$. Moreover, 
in this case $\{\Gamma h_j(x)\}_{j\in\I_{d(x)}}$ is a ONB of $J_\cW(x)$ for a.e. $x\in\T^k$.
\qed
\end{teo}

\begin{lem}\label{lem spect represent ese}
Let $\cW$ be a FSI subspace in $L^2(\R^k)$ with $d=\supes_{x\in \T^k} d(x)$, where $d(x)=\dim J_{\cW}(x)$ for $x\in\T^k$. Let $S\in L(L^2(\R^k))$ be a SP self-adjoint operator such that $R(S)\subseteq\cW$. Then, there exist:
\begin{enumerate}
\item measurable vector fields 
$v_j:\T^k\rightarrow \ell^2(\Z^k)$ for $j\in\I_{d}$ such that $v_j(x)=0$ if $j>d(x)$ and $\{v_j(x)\}_{j\in\I_{d(x)}}$ is an ONB for $J_\cW(x)$ for a.e. $x\in\T^k$;
\item bounded, measurable functions $\la_j:\T^k\rightarrow \R_+$ for $j\in\I_{d}\,$, such that $\la_1\geq \ldots\geq\la_d$, $\la_j(x)=0$ if $j>d(x)$ and 
\beq\label{lem repre espec S}
[S]_x  
=\sum_{j\in\I_{d(x)}}\lambda_j(x)\ v_j(x)\otimes v_j(x) 
\ , \peso{for a.e.}\  x\in\T^k\,.
\eeq
\end{enumerate}
\end{lem}
\begin{proof}
By considering a convenient finite partition of $\T^k$ into measurable sets we can assume, without loss of generality, that $d(x)=d$ for a.e. $x\in \T^k$.
In this case, by Theorem \ref{teo:la descom de Bo} we have that
$$
\cW=\bigoplus_{j\in \I_{d}} \cS(h_j) \ ,
$$ 
where $h_j\in \cW$, for $j\in\I_{d}$, are such that $\{\Gamma h_j(x)\}_{j\in\I_{d}}$ is a ONB  of $J_{\cW}(x)$ for a.e. $x\in \T^k$.
Consider the measurable field of self-adjoint matrices $M(\cdot):\T^k\to \matsad$ given by 
$$ 
M(x)=\big(\langle [S]_{x} \,\ \Gamma h_j(x) ,\,\ \Gamma h_i(x)\rangle\big)_{i,j\,\in \I_{d}}\ .
$$
By \cite{RS95}, we can consider measurable functions $\la_j:\T^k\to \R_+$ for $j\in \I_d\,$,
such that $\la_1\geq \ldots\geq \la_d$ and measurable vector fields $w_j:\T^k\to \C^d$ for $j\in \I_{d}\,$, such that
$\{w_j(x)\}_{j\in \I_{d}}$ is a ONB of $\C^d$ and 
\beq\label{ecuac agreg1}
M(x)=\sum_{j\in \I_{d}} \la_j(x)\,\ w_j(x) \otimes w_j(x) \peso{for a.e.} \ x\in \T^k\, .
\eeq
If $w_j(x)=(w_{ij}(x))_{i\in\I_{d}}$ for $j\in\I_{d}\,$, consider the measurable vector 
fields $v_j:\T^k \to \ell^2(\Z^k)$ for $j\in \I_{d}\,$, given by 
$$
v_j(x) =\sum_{i\in \I_{d}} w_{ij}(x)\,\ \Gamma h_i(x)\,\ \text{for}\,\ x\in \T^k\ .
$$ 
Then, it is easy to see that $\{v_j(x)\}_{j\in \I_{d}}$ is ONB of $J_\cW(x)$ for a.e. $x\in\T^k$; moreover, 
Eq. \eqref{ecuac agreg1} implies that Eq. \eqref{lem repre espec S} holds in this case.
\end{proof}

\begin{rem}\label{rem sobre estruc fina}
Let $\cW$ be a FSI subspace with $d(x)=\dim J_\cW(x)$ for $x\in\T^k$, and let $\cF=\{f_i\}_{i\in\In}\in\cW^n$ be a finite sequence in $L^2(\R^k)$ such that $E(\cF)$
is a Bessel sequence. In what follows we consider: 
\ben
\item the {\it fine spectral structure of} $E(\cF)$, that is the weakly measurable function
$$
\T^k\ni x\mapsto (\lambda_j([S_{E(\cF)}]_x )\,)_{j\in\N}\in \ell^1_+(\Z^k) \ ,
$$ 
with $\lambda_j([S_{E(\cF)}]_x  )=\la_j(x)$ as in Lemma \ref{lem spect represent ese} for $j\in\I_{d(x)}\,$, and $\lambda_j([S_{E(\cF)}]_x )=0$ for $j\geq d(x)+1$
 and $x\in \T^k$. Thus, the fine spectral structure of $\cF$
describes the eigenvalues of the positive finite rank operator $[S_{E(\cF)}]_x  =S_{\Gamma \cF(x)}\in L(\ell^2(\Z^k))$, counting multiplicities and arranged in non-increasing order.
\item The {\it fine structure of} $E(\cF)$ given by the fine spectral structure together with the measurable vector valued function $\T^k\ni x\mapsto (\|\Gamma f_i(x)\|^2)_{i\in\In}\in\R_+^n\,$.
\EOE
\een
\end{rem}

\pausa
In order to state our main result of this section we shall need the notion of vector majorization from matrix analysis. Recall that 
given $a=(a_i)_{i\in \In}\in \R^n$ and $b=(b_i)_{i\in \In}\in \R^m$ we say that $a$ is majorized by $b$, denoted $a\prec b$, if
\beq\label{eq: mayo dif}
\sum_{i\in \I_k} a_i \leq \sum_{i\in \I_k} b_i\ , \ \ 1\leq k \leq \min\{n, m\} \ \ \text{and} \ \ \sum_{i\in \I_n} a_i = \sum_{i\in \I_m} b_i\,.
\eeq

\begin{teo}[Existence of shift generated sequences with prescribed fine structure]
\label{teo sobre disenio de marcos}
 Let $\cW$ be a FSI subspace in $L^2(\R^k)$ and let $d(x)=\dim J_{\cW}(x)$ for $x\in\T^k$. 
Given measurable functions $\al_j:\T^k \to \R_+$ for $j\in \In$ and $\la_j:\T^k \to \R_+$ for $j\in \N$, 
 the following conditions are equivalent:
\ben
\item There exists $\cF=\{f_j\}_{j\in \In} \in \cW^n$ such that $E(\cF)$ is a Bessel sequence and:
\ben
\item $\|\Gamma f_j(x)\|^2=\al_j(x)$ for $j\in \In$ and a.e. $x\in\T^k$;  
\item $\la_j([S_{E(\cF)}]_x ) =\la_j(x)$ for $j\in\N$ and a.e. $x\in\T^k$.  
\een
\item The following admissibility conditions hold:
\ben
\item  $\lambda_j(x)=0$ for a.e. $x\in\T^k$ such that $j\geq\min\{d(x),n\}+1$.
\item  $(\alpha_j(x))_{j\in\In}\prec (\la_j(x))_{j\in\I_{d(x)}}$ for a.e. $x\in\T^k$.
\een
\een
\end{teo}

\pausa
Our proof of Theorem \ref{teo sobre disenio de marcos}  is based on the following extension of a basic result
in matrix analysis related with the Schur-Horn theorem (for its proof, see section \ref{Appendixity} - Appendix). In what follows we 
let $D_b\in \mat$ be the diagonal matrix with main diagonal $b\in\C^d$.

\begin{teo}\label{teo:mayo equiv} \rm
Let $b:\T^k\rightarrow (\R_+)^d$ and $c:\T^k\rightarrow (\R_+)^n$ be measurable vector fields.
The following statements are equivalent:
\ben
\item For a.e. $x\in \T^k$ we have that $c(x)\prec b(x)$.
\item There exist measurable vector fields $u_j: \T^k\to \C^d$ for $j\in\In$ such that $\|u_j(x)\|=1$ for a.e. $x\in \T^k$ 
and  $j\in \I_n\,$, and such that  
\beq
D_{b(x)}=\sum_{j\in \I_n} c_j(x)\,\ u_j(x) \otimes u_j(x) \ ,  \peso{for a.e.} \ x\in \T^k\ . \QEDP
\eeq
\een
\end{teo}

\begin{proof}[Proof of Theorem \ref{teo sobre disenio de marcos}]
Assume  that there exists $\cF=\{f_j\}_{j\in \In} \in \cW^n$ such that $\|\Gamma f_j(x)\|^2=\al_j(x)$, for $j\in \In\,$, and  
$\la_j([S_{E(\cF)}]_x ) =\la_j(x)$ for $j\in\N$ and a.e. $x\in\T^k$.
Consider the measurable field of positive semi-definite matrices $G:\T^k\to \cM_n(\C)^+$ given by the Gramian 
$$
G(x)=\Big(\, \big\langle \Gamma f_i(x)\coma  \Gamma f_j(x) \big\rangle\, \Big)_{i,j\in \In} \ , \peso{for} x\in \T^k \ . 
$$
Notice that $G(x)$ is the matrix  representation of $T^*_{\Gamma \cF(x)}T_{\Gamma \cF(x)}\in L(\C^n)$ with respect to the canonical basis of $\C^n$ for $x\in \T^k$; using the fact that the finite rank operators $T^*_{\Gamma \cF(x)}T_{\Gamma \cF(x)}$ and $T_{\Gamma \cF(x)}T^*_{\Gamma \cF(x)}=[S_{E(\cF)}]_x $ have the same positive eigenvalues (counting multiplicities) we see that 
$$
\lambda_j(G(x)) = \begin{cases}  \lambda_j(x) &  \peso{for} 1\leq j\leq \min\{d(x),n\} \\ 
\ \ \ 0 & \peso{for} \min\{d(x),n\}+1\leq j\leq n \end{cases} \peso{for a.e.} x\in\T^k  \ . 
$$
On the other hand,  the main diagonal of $G(x)$ is $(\|\Gamma f_j(x)\|^2)_{j\in\In}=(\al_j(x))_{j\in\In}\,$; 
hence, by the classical Schur-Horn theorem (see \cite{HJ13}) we see that 
$$
(\alpha_j(x))_{j\in\In}\prec \lambda(G(x))\in \R_+ ^n  \implies 
(\alpha_j(x))_{j\in\In}\prec (\lambda_j(x))_{j\in \I_{d(x)}}  \peso{for a.e.} x\in\T^k \ .
$$

\pausa
Conversely, assume that $(\al_j(x))_{j\in \In} \prec (\la_i(x))_{i\in \I_{d(x)}}$ for a.e. $x\in \T^k$. 
By considering a convenient finite partition of $\T^k$ into measurable subsets we can assume, without loss of generality, that $d(x)=d$ for $x\in\T^k$. Therefore, by Theorem \ref{teo:mayo equiv}, there exist
measurable vector fields $u_j: \T^k\to \C^d$ for $j\in\In$ such that $\|u_j(x)\|=1$ for a.e. $x\in \T^k$ and  $j\in \I_n$, and such that  
\beq \label{diago}
D_{\lambda(x)}=\sum_{j\in \I_n} \alpha_j(x)\,\ u_j(x) \otimes u_j(x) \ ,  \peso{for a.e.} \ x\in \T^k\ ,
\eeq
where $\lambda(x)=(\lambda_j(x))_{j\in\I_d}$ for $x\in \T^k$. Now, by Theorem \ref{teo:la descom de Bo} there exist measurable vector fields 
$v_j:\T^k\rightarrow \ell^2(\Z^k)$ for $j\in\I_d$ such that $\{v_j(x)\}_{j\in\I_d}$ is a ONB of $J_\cW(x)$ for a.e. $x\in\T^k$.
Let $u_j(x)=(u_{ij}(x))_{i\in\I_d}$ for $j\in\In$ and $x\in\T^k$; then we consider 
the finite sequence $\cF=\{f_j\}_{j\in\In}\in\cW^n$ determined by 
$\Gamma f_j(x)=\alpha_j^{1/2}(x) \sum_{i\in\I_d} u_{ij}(x) \ v_i(x)$ for $j\in\In$ and $x\in\T^k$.
It is clear that 
$$
\|\Gamma f_j(x)\|^2=\|\alpha_j^{1/2}(x)\ u_j(x)\|^2=\alpha_j(x) \peso{ for a.e. } x\in\T^k,\quad j\in\In\ .
$$ 
Moreover, using Eq. \eqref{diago} it is easy to see that 
$$\left(\sum_{j\in\In} \Gamma f_j(x)\otimes \Gamma f_j(x) \right)\, v_i(x)= \la_i(x)\,v_i(x) \peso{for} i\in\I_d \ \ \text { and  \ a.e. }\  x\in\T^k\,.  $$
Hence, $\la_j([S_{E(\cF)}]_x ) =\la_j(x)$ for $j\in\N$ and a.e. $x\in\T^k$
\end{proof}

\begin{rem}\label{se puede incluso con op SP}
 Let $\cW$ be a FSI subspace in $L^2(\R^k)$ and let $d(x)=\dim J_{\cW}(x)$ for $x\in\T^k$. 
Let $\al_j:\T^k \to \R_+$, $j\in \In\,$, be measurable functions and let $S\in L(L^2(\R^k))^+$ be a positive SP operator 
such that $R(S)\subseteq \cW$. Let $\T^k\ni x\mapsto (\la_j([S]_x) )_{j\in\N}$ be the fine spectral structure of $S$ (which is well defined by Lemma 
\ref{lem spect represent ese}).
Assume that for a.e. $x\in\T^k$ we have that 
$$ 
(\alpha_j(x))_{j\in\In}\prec (\la_j([S]_x) )_{j\in\I_{d(x)}}\ .
$$
Then, there exists $\cF=\{f_j\}_{j\in \In} \in \cW^n$ such that $E(\cF)$ is a Bessel sequence, 
$$
S_{E(\cF)}=S \py \|\Gamma f_j(x)\|^2=\alpha_j(x) \peso{for a.e.}  x\in\T^k \coma   j\in\In \ . 
$$
Indeed, if in the proof of Theorem \ref{teo sobre disenio de marcos} above we take the measurable vector fields 
$v_j:\T^k\rightarrow \ell^2(\Z^k)$ such that $\{v_j(x)\}_{j\in\I_d}$ is a ONB of $J_\cW(x)$ and such 
that $[S]_x \, v_j(x)=\la_j([S]_x)\ v_j$ for a.e. $x\in\T^k$ (notice that this can always be done by Lemma \ref{lem spect represent ese})
then we conclude, as before, that 
\beq
[S_{E(\cF)}]_x \ v_j(x)= \la_j([S]_x)\,v_j(x) \peso{for} j\in\I_d \implies [S_{E(\cF)}]_x  = [S]_x \peso{for a.e.} x\in\T^k\ . \EOEP
\eeq
\end{rem}

\pausa
As a first application of Theorem \ref{teo sobre disenio de marcos} we show the existence of shift 
generated uniform  tight frames for an arbitrary FSI. In turn, this allows 
us to strengthen some results from \cite{BMS15} (see also Corollary \ref{coro tight 2}). 

\begin{cor}\label{cororo1}
Let $\{0\}\neq\cW\subset L^2(\R^k)$ be a FSI subspace and let $d(x)=\dim J_\cW(x)$ for $x\in\T^k$. Assume that $n\geq \supes_{x\in\T^k}d(x)$, 
let $Z_i=d^{-1}(i)$ for $i\in\In\,$ and set $C_\cW= \sum_{i\in\In} i\cdot |Z_i|>0$. Then:
\ben
\item There exists a sequence $\cF=\{f_j\}_{j\in\In}\in\cW^n$ such that $\|f_j\|^2=n^{-1}$ for $j\in\In$ 
and such that $E(\cF)$ is a uniform tight frame for $\cW$.
\item For any sequence $\cG=\{g_j\}_{j\in\In}\in\cW^n$ such that $E(\cG)$ is a Bessel sequence and such that $\sum_{j\in\In} \|g_j\|^2=1$, and for every 
$\varphi\in\convf$ we get that: 
\beq\label{del item 2}
P^\cW_\varphi(E(\cG))\geq C_\cW \ \varphi(C_\cW^{-1})=P^\cW_\varphi(E(\cF))\, .
\eeq
Moreover, if we assume that $\varphi\in\convfs$ then $P^\cW_\varphi(E(\cG))=C_\cW \ \varphi(C_\cW^{-1})$ if and only 
if $E(\cG)$ is a tight  frame for $\cW$. 
\een 
\end{cor}
\begin{proof}
Let $p_i=|Z_i|$ (where $|A|$ denotes the Lebesgue measure of $A\subset \T^k$ and $Z_i=d^{-1}(i)$) for $1\leq i\leq n$; 
then $C_\cW= \sum_{i\in\In} i\cdot p_i$. Notice that by hypothesis Spec $\cW=\cup_{i\in\In} Z_i\,$. 
For $x\in\T^k$ set $\alpha(x)=0$ if $x\in \T^k\setminus\text{Spec}(\cW)$ and: 
\beq \label{defi alfas}
\alpha(x):=\left\{
  \begin{array}{ccc}
    \frac{j\cdot C_\cW^{-1} }{n}
		& {\rm if} & x\in Z_j \ \text{ and } \ p_j>0\,; \\
   0 & {\rm if} & x\in Z_j \ \text{ and } \ p_j=0 \, .\\
   \end{array}
	\right.
\eeq 
Then, it is easy to see that $(\alpha(x))_{i\in\In}\prec (C_\cW^{-1})_{i\in\I_{d(x)}}$ for every $x\in \T^k$; hence, by Theorem \ref{teo sobre disenio de marcos} we see that there exists $\cF=\{f_j\}_{j\in\In}\in\cW^n$ such that $\|\Gamma f_j(x)\|^2=\alpha(x)$ for $j\in\In$ and such that $S_{\Gamma \cF(x)}=C_\cW^{-1}\ P_{J_\cW(x)}$ for a.e. $x\in\T^k$. Therefore, $S_{E(\cF)}=C_\cW^{-1}\, P_\cW$ and 
$$
\|f_j\|^2=\int_{\T^k} \alpha(x)\ dx=
\frac{C_\cW^{-1}}{n} \, \sum_{i\in\In}  \int_{Z_i} i\ dx=\frac{C_\cW^{-1}}{n}\ \sum_{i\in\In} i \cdot p_i=\frac{1}{n}\ .
$$
If $\cG$ is as in item 2. then, by \cite{BMS15}, we get the inequality \eqref{del item 2}. 
Notice that the lower bound is attained at $\cF$ (since it is tight); 
the last part of the statement was already shown in \cite{BMS15}. 
\end{proof}

\subsection{Generalized (measurable) eigensteps}\label{subsec eigensteps}

In this section we derive a natural extension of the notion of eigensteps introduced in \cite{CFMPS}, that allows us to describe a procedure to inductively construct finite sequences $\cF=\{f_i\}_{i\in\In}\in\cW^n$ such that the fine structure of $E(\cF)$ (that is, the fine
spectral structure of $E(\cF)$ and the finite sequence of measurable functions 
$\|\Gamma f_i(\cdot)\|^2:\T^k\rightarrow \R_+$, $i\in\I_n$) are prescribed. Hence, we obtain an in-depth description of a step-by-step construction of Bessel sequences $E(\cF)$ with prescribed fine structure. We point out that our techniques are not based on those from \cite{CFMPS}; indeed, our approach is based on an additive model developed in \cite{BMS15}.

\begin{rem}\label{rem hay eigensteps}Let $\cW$ be a FSI subspace and let $d(x)=\dim J_\cW(x)$ for $x\in\T^k$;
let $\cF=\{f_i\}_{i\in\In}\in\cW^n$ be such that $E(\cF)$ is a Bessel sequence and set:
\ben 
\item $\alpha_i(x)=\|\Gamma f_i(x)\|^2$, for $x\in\T^k$ and $i\in\In\,$.
\item $\la_i(x)=\lambda_i([S_{E(\cF)}]_x )$ , for $x\in\T^k$ and $i\in\In\,$,
where $\T^k\ni x\mapsto (\lambda_i([S_{E(\cF)}]_x )\,)_{i\in\N}$ denotes the fine spectral structure of $E(\cF)$.
\een
By Theorem \ref{teo sobre disenio de marcos} these functions satisfy the following admissibility conditions:
\ben
\item[Ad.1]  $\lambda_i(x)=0$ for a.e. $x\in\T^k$ such that $i\geq\min\{n,\,d(x)\}+1$.
\item[Ad.2]  $(\alpha_i(x))_{i\in\In}\prec (\la_i(x))_{i\in\I_{d(x)}}$ for a.e. $x\in\T^k$.
\een
For $j\in\In$ consider the sequence $\cF_j=\{f_i\}_{i\in\I_{j}}\in\cW^j$. In this case $E(\cF_j)
=\{T_\ell f_i\}_{(\ell,i)\in\Z^k\times \I_{j}}$ is a Bessel sequence and $S_j=S_{E(\cF_j)}$ 
is a SP operator such that 
$$
[S_j]_x=S_{\Gamma \cF_j(x)}
=\sum_{i\in\I_{j}} \Gamma f_i(x)\otimes \Gamma f_i(x)\in L(\ell^2(\Z^k))^+ \peso{for a.e. } x\in\T^k\ , \ \  j\in\In\ .
$$ 
For $j\in\In$ and $i\in\I_{j}\,$, consider the measurable function $\lambda_{i,j}:\T^k\rightarrow \R_+$ given by 
$$
\lambda_{i,j}(x)=\lambda_i([S_j]_x)\peso{for} x\in\T^k\  ,
$$
where $\T^k\ni x\mapsto (\lambda_i([S_j]_x))_{i\in\N}$ denotes the fine spectral structure of $E(\cF_j)$ (notice that by construction $\lambda_i([S_j]_x)=0$ for $i\geq j+1$). Then, it is well known (see \cite{CFMPS}) that 
 $(\lambda_{i,j}(x))_{i\in\I_{j}}$ interlaces $(\lambda_{i,(j+1)}(x))_{i\in\I_{j+1}}$ i.e.
$$
\lambda_{i,(j+1)}(x)\geq \lambda_{i,j}(x)\geq \lambda_{(i+1),(j+1)}(x) \peso{for} i\in\I_{j}\ ,\ j\in\I_{n-1}\ , \ \text{ and a.e. } x\in\T^k\ . 
$$
Notice that for a.e. $x\in\T^k$, 
$$
\sum_{i\in\I_{j}} \lambda_{i,j}(x)=   \tr([S_j]_x)
=  \sum_{i\in\I_{j}}\|\Gamma f_i(x)\|^2 =\sum_{i\in\I_{j}}\alpha_i(x) \peso{for} j\in\In\ . 
$$
Finally notice that by construction $S_n=S_{E(\cF)}$ and hence, $\lambda_{i,n}(x)=\lambda_i(x)$ for $i\in\In$ and $x\in \T^k$.
These facts motivate the following extension of the notion of eigensteps introduced in \cite{CFMPS}.
\EOE
\end{rem}


\begin{fed}\label{defi measiegen} \rm 
Let $\cW$ be a FSI subspace and let $\la_i\coma  \alpha_i:\T^k\rightarrow \R_+$ for $i\in\In$ be measurable functions satisfying the admissibility assumptions Ad.1 and Ad.2 in Remark \ref{rem hay eigensteps}. 
 A sequence of eigensteps for $(\lambda,\alpha)$ is a doubly-indexed sequence 
of measurable functions $\lambda_{i,j}:\T^k\rightarrow \R_+$ for $i\in\I_{j}$ and $j\in\In$ such that:
\ben
\item $\lambda_{i,(j+1)}(x)\geq \lambda_{i,j}(x)\geq \lambda_{(i+1),(j+1)}(x)$ for $i\in\I_{j} \coma j\in\I_{n-1}\coma $
 and a.e. $x\in\T^k$; 
\item $\sum_{i\in\I_{j}} \lambda_{i,j}(x)= \sum_{i\in\I_{j}}\alpha_i(x)$ for $j\in\In$ and a.e. $x\in \T^k$;
\item $\lambda_{i,n}(x)=\lambda_i(x)$ for $i\in\In$ and a.e. $x\in \T^k$. \EOE
\een
\end{fed}

\begin{rem}\label{rem hay eigensteps2}
Consider the notations and terminology from Remark \ref{rem hay eigensteps}. Then $((\lambda_{i,j}(\cdot))_{i\in\I_{j}})_{j\in\In}$ is a sequence of eigensteps for $(\lambda,\alpha)$. We say that $((\lambda_{i,j}(\cdot))_{i\in\I_{j}})_{j\in\In}$ is the sequence of eigensteps for $(\la,\alpha)$ associated to $\cF$. \EOE
\end{rem}

\pausa
In what follows we show that every sequence of eigensteps is associated to some $\cF=\{f_i\}_{i\in\In}\in\cW^n$ such that $E(\cF)$ is a Bessel sequence (see Theorem \ref{teo todo eigen echachi} below). In order to show this, we recall an additive (operator) model from \cite{BMS15}.

\begin{fed}\label{el conjunto U}\rm
Let $\cW$ be a FSI subspace and let $d:\T^k\rightarrow \N_{\geq 0}$ be the measurable function given by $d(x)=\dim J_{\cW}(x)$ for $x\in\T^k$. Let $S\in L(L^2(\R^k))^+$ be SP and such that $R(S)\subset \cW$. Given a measurable function $m:\T^k\rightarrow \Z$ such that $m(x)\leq d(x)$ for a.e. $x\in\T^k$ we consider
$$U^\cW_m(S) = 
\Big\{S+ B:B\in L(L^2(\R^k))^+ \text{ is SP},\, R(B)\subset \cW,\,
\rk([B]_x)\leq d(x)-m(x) \ \text{for a.e. } x\in \T^k  \Big\}  \,.
$$ \EOE
\end{fed}

\begin{teo}[Appendix of \cite{BMS15}]\label{estructdelU} \rm
Consider the notations from Definition \ref{el conjunto U}. Given a measurable function $\mu:\T^k\rightarrow \ell^1(\N)^+$ the following are equivalent:
 \ben
 \item There exists $C\in U^\cW_m(S)$ such that $\la(\hat C_x)=\mu(x)$, for a.e. $x\in \T^k$;
 \item For a.e. $x\in\T^k\setminus \text{Spec}(\cW)$ then $\mu(x)=0$; for a.e. $x\in \text{Spec}(\cW)$ we have that  $\mu_i(x)=0$ for $i\geq d(x)+1$ and 
 \ben 
 \item in case $m(x)\leq 0$, $\mu_i(x)\geq \la_i([S]_ x)$ for $i\in\I_{d(x)}\,$;
 \item in case $m(x)\in\I_{d(x)}\,$, $\mu_i(x)\geq \la_i([S] _x)$ for $i\in\I_{d(x)}\,$ and
\beq
\la_i([S]_x)\geq \mu_{(d(x)-m(x))+i}(x) \peso{for} i\in\I_{m(x)}\ . \QEDP
\eeq
  \een
 \een
\end{teo}

\begin{rem}
We point out that Theorem \ref{estructdelU} is obtained in terms of a natural extension of the Fan-Pall interlacing theory from matrix theory, to the context of measurable fields of positive matrices (see \cite[Appendix]{BMS15}); we also notice that the result is still valid for fields (of vectors and operators) defined in measurable subsets of $\T^k$.  The original motivation for considering the additive model above was the fact that it describes the set of frame operators of oblique duals of a fixed frame. In the present setting, this additive model will also allow us to link the sequences of eigensteps with the construction of SG Bessel sequences with prescribed fine structure. \EOE
\end{rem}

\begin{teo}\label{teo todo eigen echachi}
Let $\cW$ be a FSI subspace and let $\la_i,\, \alpha_i:\T^k\rightarrow \R_+$ for $i\in\In$ be measurable functions satisfying the admissibility conditions Ad.1 and Ad.2 in Remark \ref{rem hay eigensteps}. 
Consider a sequence of eigensteps 
$((\lambda_{i,j}(\cdot))_{i\in\I_{j}})_{j\in\In}$ for $(\la,\alpha)$. 
Then, there exists $\cF=\{f_i\}_{i\in\In}\in\cW^n$ such that $E(\cF)$ is a Bessel sequence and $((\lambda_{i,j}(\cdot))_{i\in\I_{j}})_{j\in\In}$ is the sequence of eigensteps associated to $\cF$.
\end{teo}
\begin{proof} First notice that both the assumptions as well as the properties of the objects that we want to construct are checked point-wise; hence, by considering a convenient partition of $\T^k$ into measurable sets we can assume (without loss of generality) that $d(x)=d\geq 1$, for $x\in\T^k$. Now, we argue by induction on $j$. Notice that by hypothesis for $i=j=1$, we see that $\lambda_{1,1}(x)=\alpha_1(x)$ for a.e. $x\in\T^k$.
Let $f_1\in\cW$ be such that $\|\Gamma f_1(x)\|^2=\alpha_1(x)$ for a.e. $x\in\T^k$; indeed, we can take $f_1\in\cW$ determined by the condition $\Gamma f_1(x)=\alpha^{1/2}(x)\ \Gamma h_1(x)$, where $\{h_i\}_{i\in\I_{\ell}}$ are the quasi orthogonal generators for the orthogonal sum decomposition of $\cW$ as in Theorem \ref{teo:la descom de Bo}. Then, by construction $\|\Gamma f_1(x)\|^2=\alpha_1(x)$ and $\lambda_{1,1}(x)=\|\Gamma f_1(x)\|^2=\lambda_1([S_{E(f_1)}]_x)$ for a.e. $x\in \T^k$.

\pausa
Assume that for $j\in\I_{n-1}$ we have constructed $\cF_j=\{f_i\}_{i\in\I_{j}}\in\cW^j$ such that 
\beq\label{eq induc}
\la_{i,\ell}(x)=\la_i([S_{E(\cF_\ell)}]_x) \peso{for} i\in\I_{\ell} \ , \ \ \ell\in\I_{j}\ \ \text{and a.e. } x\in\T^k\,. 
\eeq
We now construct $f_{j+1}$ as follows: set $\mu_i=\la_{i,j+1}$ for $i\in\I_{j+1}$ and $\mu_i=0$ for $i>j+1$; set $\mu=(\mu_i)_{i\in\N}:\T^k\rightarrow \ell^1(\N)^+$ which is a measurable function. Further, set $S=S_{E(\cF_j)}\in L(L^2(\R^k))^+$ which is a SP operator with $R(S)\subset \cW$ and set $m(x)=m=d-1$.
Moreover, by taking $\ell=j$ in Eq. \eqref{eq induc} above we see that $\la_i([S]_x)=\la_{i,j}(x)$ for $i\in\I_{j}$ and $\la_i([S]_x)=0$ for $i\geq j+1$, for a.e. $x\in\T^k$ .

\pausa
By hypothesis, we see that $\la_{i,j+1}\leq \la_{i,j+2}\leq \ldots\leq \la_{i,n}=\la_i$; since the admissibility conditions in Remark \ref{rem hay eigensteps} hold, we conclude that $\mu_i=\la_{i,j+1}=0$ whenever $i\geq d+1$. On the other hand, since 
$d-m=1$  we see that 
the conditions in item 2. in Theorem \ref{estructdelU} can be put together as the interlacing relations 
$$ \mu_i(x)\geq \la_i([S]_x)\geq\mu_{i+1}(x) \peso{for} i\in\I_{j} \ \text{ and a.e. }\  x\in\T^k\,, $$
which hold by hypothesis (see condition 1. in Definition \ref{defi measiegen}); therefore, by Definition \ref{el conjunto U} and Theorem \ref{estructdelU},  there exists a SP operator $B\in L(L^2(\R^k))^+$ such that $R(B)\subset \cW$, $\rk([B] _x )\leq 1$ for a.e. $x\in\T^k$ and such that $\la_i([S+B]_x)=\mu_i(x)=\la_{i,j+1}(x)$ for $i\in\I_{j+1}\,$, for a.e. $x\in\T^k$.  
The previous conditions on $B$ imply that there exists $f_{j+1}\in\cW$ such that $B=S_{E(f_{j+1})}$; indeed, $f_{j+1}$ is such that it satisfies:
$\Gamma f_{j+1}(x)\otimes \Gamma f_{j+1}(x)=[B] _x$ for a.e. $x\in\T^k$.
Finally, if we set $\cF_{j+1}=\{f_i\}_{i\in\I_{j+1}}$ then $S_{E(\cF_{j+1})}=S_{E(\cF_{j})}+S_{E(f_{j+1})}=S+B$ and hence $\la_{i,j+1}(x)=
\la_i([S_{E(\cF_{j+1})}]_x)$ for $i\in\I_{j+1}$ and a.e. $x\in\T^k$. This completes the inductive step.
\end{proof}
\pausa
We end this section with the following remark. With the notations and terminology in Theorem \ref{teo todo eigen echachi}, notice that the 
constructed sequence $\cF=\{f_i\}_{i\in\In}$ is such that 
its fine structure is prescribed by $(\lambda,\,\alpha)$: indeed, $\lambda_i([S_{E(\cF)}]_x)=\lambda_{i,\,n}(x)=\lambda_i(x)$  and 
$\|\Gamma f_i(x)\|^2=\alpha_i(x)$ for $i\in\In$ and a.e. $x\in\T^k$ (this last fact can be checked using induction and item 2. in Definition 
\ref{defi measiegen}). That is, the measurable eigensteps provide a detailed description of Bessel sequences $E(\cF)$ with prescribed fine structure.

\section{An application: optimal frames with prescribed norms for FSI subspaces}\label{sec opti frames with prescribed norms}

In order to describe the main problem of this section we consider the following:
\begin{fed}\label{defi bal} \rm
Let $\cW$ be a FSI subspace of $L^2(\R^k)$ and let $\alpha=(\alpha_i)_{i\in\In}\in (\R_{>0}^n)\da$. We let
\beq
\mathfrak {B}_\alpha(\cW)=\{\cF=\{f_i\}_{i\in\In}\in\cW^n:\ E(\cF) \ \text{is a Bessel sequence }, \ \|f_i\|^2=\alpha_i\,,\ i\in\In\}\ ,
\eeq
the set of SG Bessel sequences in $\cW$ with norms prescribed by $\al$. 
\EOE
\end{fed}

\pausa
Notice that the restrictions on the families $\cF=\{f_i\}_{i\in\In}\in \mathfrak {B}_\alpha(\cW)$
(namely $\|f_i\|^2=\alpha_i$ for $i\in\In$) are of a {\it global} nature. Our problem is to describe those $\cF\in \mathfrak {B}_\alpha(\cW)$ such that 
the encoding schemes associated to their corresponding Bessel sequences $E(\cF)$ are as stable as possible. Ideally, we would search for sequences $\cF$ such that $E(\cF)$ are tight 
frames for $\cW$; yet, Theorem \ref{teo sobre disenio de marcos} shows that there are obstructions for the existence 
of such sequences  
(see Corollary  \ref{coro tight 2} below).

\pausa
By a simple re-scaling argument, we can assume that $\sum_{i\in\In}\alpha_i=1$; then Corollary \ref{cororo1} (see also \cite[Theorem 3.9.]{BMS15}) shows that if there exists $\cF_0\in \mathfrak {B}_\alpha(\cW)$ such that $E(\cF_0)$ is a tight 
frame for $\cW$ then $E(\cF_0)$ is a minimizer in $\mathfrak {B}_\alpha(\cW)$ of every frame potential $P_\varphi^\cW$ for any convex function $\varphi\in\convf$ and $P_\varphi^\cW (E(\cF_0))= C_\cW \ \varphi(C_\cW^{-1})$; moreover, in case $\varphi\in\convfs$ is a strictly convex function, then every such $\cF\in \mathfrak {B}_\alpha(\cW)$ for which $P_\varphi^\cW(E(\cF))=C_\cW \ \varphi(C_\cW^{-1})$ is a tight frame.
This suggests that in the general case, in order to search for $\cF\in \mathfrak {B}_\alpha(\cW)$ 
such that the encoding schemes associated to their corresponding Bessel sequences $E(\cF)$ are as stable as possible, we could study the minimizers in $\mathfrak {B}_\alpha(\cW)$ of the convex potential $P_\varphi^\cW$ associated to a strictly convex function $\varphi\in\convfs$.

\pausa
Therefore, given $\varphi\in\convf$, in what follows we show the existence of finite sequences $\cF^{\rm op}\in \mathfrak {B}_\alpha(\cW)$ such that 
$$
P_\varphi^\cW(E(\cF^{\rm op}))=\min\{P_\varphi^\cW(E(\cF)): \ \cF\in \mathfrak {B}_\alpha(\cW)\}\,.
$$ Moreover, in case $\varphi\in\convfs$ then we describe the fine spectral structure of the frame operator of $E(\cF^{\rm op})$.
In case $\varphi(x)=x^2$, our results extend some results from \cite{BF,CKFT,MR10} for the frame potential to the context of 
SG Bessel sequences lying in a FSI subspace $\cW$.

\pausa
Let us fix some general notions and notation for future reference: 
\begin{notas}\label{nota impor2} 
 In what follows we consider:
\ben
\item  A FSI subspace $\cW\subset L^2(\R^k)$; 
\item $d(x)=\dim J_\cW(x)\leq \ell\in\N$, for a.e. $x \in \T^k$;
\item The Lebesgue measure on $\R^k$, denoted $|\cdot|$ ; $Z_i=d^{-1}(i)\subseteq \T^k$ and $p_i=|Z_i|$, $i\in\I_{\ell}\,$.
\item The spectrum of $\cW$ is the measurable set $\text{Spec}(\cW) =
\bigcup_{i\in\I_{\ell}} Z_i = \{x\in \T^k: d(x)\neq 0\}$. 
\een\EOE
\end{notas}

\subsection{The uniform dimension case}\label{subse uniform}

Consider the Notations \ref{nota impor2}. In this section we obtain the fine spectral structure of minimizers of convex potentials in $\mathfrak B_\alpha(\cW)$ under the assumption that $d(x)=d$ for a.e. $x\in\text{Spec}(\cW)$. In order to deal with this particular case, we recall some notions and constructions from \cite{BMS15}.

\begin{rem}[Waterfilling in measure spaces]\label{recordando waterfilling}
Let $(X,\cX,\mu)$ denote a probability space and let $L^\infty(X,\mu)^+ 
= \{g\in L^\infty(X,\mu): g\ge 0\}$. Recall that for  
$f\in L^\infty(X,\mu)^+ $ and $c \geq \infes f\geq 0$ we consider the {\it waterfilling} of $f$ at level $c$, denoted $f_c\in L^\infty(X,\mu)^+$, given by $f_c= \max\{f,\,c\}=f + (c-f)^+$, where $g^+$ denotes the positive part of a real function $g$. 
Recall the decreasing rearrangement of non-negative functions defined in Eq. \eqref{eq:reord}. It is straightforward to check that if 
\beq\label{eq lem reord}
s_0= \mu\{x\in X:\ f(x)>c\}  \peso{then}
f_c^*(s)=\left\{
  \begin{array}{ccc}
     f^*(s) & if & 0\leq s <s_0 \,; \\
      c & if & s_0\leq s \leq 1\, .
   \end{array}
	\right.
	\eeq
We further consider $\phi_f: [\infes f, \infty)\to \R_+$ given by
$$
\phi_f(c)=\int_X f_c\ d\mu= \int_X f(x) + (c-f(x))^+\ d\mu(x)\,. $$ 
Then, it is easy to see that: 
\ben
\item $\phi_f(\infes f)=\int_X f\ d\mu$ and $\lim _{c\to +\infty} \phi_f(c)= +\infty$;
\item $\phi_f$ is continuous and strictly increasing.
\een Hence, for every $v\geq \int_X f\ d\mu$ there exists a unique $c=c(v)\geq \infes f$ such that 
$\phi_f(c)=v$. 
With the previous notations then, by \cite[Theorem 5.5.]{BMS15} we get that if $h\in L^\infty(X,\mu)^+ $ is such that 
\beq\label{eq teo 5.5}
f\leq h  \py  v\leq \int_X h\ d\mu \peso{then} f_{c(v)}\prec_w h\ .\eeq 
\EOE
\end{rem}

\begin{lem}\label{lem: wat er filling}
Let $(X,\cX,\mu)$ denote a probability space and let $f,\,g\in L^\infty(X,\mu)^+$ be such that $f\prec_w g$. Let $c,\,d\geq 0$ be such that 
$\int_X f_c\ d\mu=\int_X g_d\ d\mu$, where $f_c$ and $g_d$ denote the waterfillings of $f$ and $g$ at levels $c$ and $d$ respectively. 
Then $f_c\prec g_d$ in $(X,\mu)$. 
\end{lem}
\begin{proof} Set $s_0=\mu\{x\in X:\ f(x)>c\}\in[0,1]$; notice that by construction $g\leq g_d$ in $X$ so 
that, by Remark \ref{rem:prop rear elem}, $g^*\leq (g_d)^*$ in $[0,1]$. Hence, for every $s\in[0,s_0]$ we have 
\beq\label{eq desi11}
\int_0^s(f_c)^*\ dt=\int_0^s f^*\ dt\leq \int_0^s g^*\ dt\leq \int_0^s (g_d)^*\ dt\,.
\eeq
On the other hand,
$$ \int_0^{s_0} (g_d)^* \ dt\geq \int_0^{s_0} g^* \ dt\geq \int_0^{s_0} f^* \ dt \implies 
\omega \igdef\int_0^{s_0} (g_d)^* \ dt-\int_0^{s_0} f^* \ dt\geq 0\, .$$ 
Using Remark \ref{recordando waterfilling} and the hypothesis we get that
$$ 
\int_0^{s_0} f^* \ dt + (1-s_0) \, c = \int_0^1 f_c^* \ dt 
\stackrel{\eqref{reor int}}{=}\int_X f_c\ d\mu=\int_X g_d\ d\mu
\stackrel{\eqref{reor int}}{=}   \int_0^{s_0} (g_d)^* \ dt + \int_{s_0}^1 (g_d)^* \ dt 
$$
$$
\implies \quad \quad (1-s_0) \, c=\int_{s_0}^1 \Big[ \ (g_d)^* + \frac{\omega}{1-s_0} \ \Big]\ dt\,.$$
Thus, by \cite[Lemma 5.3.]{BMS15} we get that for $s\in[s_0\coma 1]$:
$$ 
(s-s_0)\,c\leq \int_{s_0}^s \Big[ \ (g_d)^* + \frac{\omega}{1-s_0} \ \Big]\ dt
\leq \int_{s_0}^s (g_d)^* \ dt+ \omega\  .
$$
This last identity and Remark \ref{recordando waterfilling} show that for $s\in[s_0\coma 1]$,
\beq \label{eq desi22}
\int_0^s (f_c)^*\ dt=
\int_0^{s_0} (g_d)^*\ dt-\omega + (s-s_0)\, c\leq \int_0^{s_0} (g_d)^*\ dt + \int_{s_0}^s (g_d)^*\ dt\,.
\eeq
The lemma is a consequence of Eqs. \eqref{eq desi11} and \eqref{eq desi22}.
\end{proof}

\begin{rem}\label{rem: sobre medidas}
Let $(Z\coma \mathcal Z\coma |\, \cdot \,|)$ be a (non-zero) measure subspace of $(\T^k\coma \mathcal B\coma |\, \cdot \,|)$ 
and consider $(\I_r, \mathcal P(\I_r),\#(\cdot))$ i.e, $\I_r$ endowed with the counting measure. In what follows we consider the product 
space $X\igdef Z\times \I_r$ endowed with the product measure $\mu\igdef|\cdot |\times \#(\cdot)$.
\EOE
\end{rem}

\begin{lem}\label{lem utilisima} Consider the notations in Remark \ref{rem: sobre medidas} and 
let $\alpha: Z\rightarrow \R^r$ be a measurable function. 
Let $\breve\alpha:X\rightarrow \R$ be given by 
$$ 
\breve \alpha(x,i)=\alpha_i(x) \peso{for} x\in Z \py i\in\I_r\ .
$$
Then $\breve \alpha$ is a measurable function and we have that:
\ben
\item If $\varphi\in\convf$ then 
$\int_X \varphi\circ \breve \alpha \ d\mu = \suml_{i\in\I_r} \int_{Z} \varphi(\alpha_i(x))\ dx \ .$
\item Let $\beta: Z\rightarrow \R^r$ be a measurable function and let $\breve \beta:X\rightarrow \R$ 
be constructed analogously. If 
$$
\alpha(x)\prec\beta(x) \peso{for a.e.} x\in Z   \ \implies  \ \ \breve \alpha\prec \breve\beta 
$$
in the probability space $(X,\mathcal X,\tilde \mu)$, where $\tilde \mu=(r\cdot|Z|)^{-1}\,\mu$. 
\item 
Similarly, $\alpha(x)\prec_w\beta(x)$ for a.e. $x\in Z$ implies that 
$\breve \alpha\prec_w \breve \beta$ in $(X,\mathcal X,\tilde \mu)$.
\een
\end{lem}
\begin{proof}
The proof of the first part of the statement is straightforward. In order to see item 2., notice that if $\varphi\in\convf$ then $\alpha(x)\prec\beta(x)$ implies that $\sum_{i\in\I_r}\varphi(\alpha_i(x))\leq  \sum_{i\in\I_r}\varphi(\beta_i(x))$ for a.e. $x\in Z$. Then, using item 1. we get that
$$
\int_X \varphi\circ \breve \alpha\ d\tilde \mu=(r\cdot |Z|)^{-1} \int_{Z} \sum_{i\in\I_r}\varphi(\alpha_i(x)) \ dx
\leq (r\cdot |Z|)^{-1} \int_{Z} \sum_{i\in\I_r}\varphi(\beta_i(x))\ dx = \int_X \varphi\circ \breve \beta\ d\tilde \mu\ .
$$ 
Since $\varphi\in\convf$ is arbitrary, Theorem \ref{teo porque mayo} shows that 
$\breve \alpha\prec \breve \beta$. Item 3.  follows using similar arguments, based on the characterization of submajorization in terms of 
integral inequalities involving non-decreasing convex functions given in Theorem \ref{teo porque mayo} (see also \cite{Chong}).
\end{proof}


\pausa
The following is the first main result of this section.

\begin{teo}[Existence of optimal sequences in $\mathfrak {B}_\alpha(\cW)$]\label{teo dim unif} 
Consider the Notations \ref{nota impor2}. Let $\alpha=(\alpha_i)_{i\in\In}\in (\R_{>0}^n)\da$ and assume that $\cW$ is such that $d(x)=d$ for a.e. $x\in \text{Spec}(\cW)$; set $r=\min\{n,d\}$. Let $p=p_d=|\text{Spec}(\cW)|$. Then there exist $c=c(\alpha,\,d,\,p)\geq 0$ and $\cF^{\rm op}\in \mathfrak {B}_\alpha(\cW)$ such that:
\ben
\item For a.e. $x\in \text{Spec}(\cW)$ we have that
\beq\label{eq defi la op unif dim}
\la_j([S_{E(\cF^{\rm op})}]_x)  
=\left\{
  \begin{array}{ccc}
     \max\{\frac{\alpha_j}{p}\, , \, c\} & if & j\in\I_{r} \ ; \\
      0 & if & r+1 \le j\le d\  .
   \end{array}
	\right.
	\eeq
	In particular, if $d\leq n$ (i.e. $r=d$) then $E(\cF^{\rm op})$ is a frame for $\cW$.
\item For every $\varphi\in\convf$ and every $\cF\in \mathfrak {B}_\alpha(\cW)$ then 
\beq\label{eq prop c}
 p\cdot \sum_{j\in\I_{r}}  \varphi (\max\{\frac{\alpha_j}{p}\, , \, c\}) + p \, (d-r)\, \varphi(0) =P_\varphi^\cW(E(\cF^{\rm op}))\leq P_\varphi^\cW(E(\cF))
\,.
\eeq
\een 
\end{teo}
\begin{proof}
Consider $\text{Spec}(\cW)$ as a (non-zero, otherwise the result is trivial) 
measure subspace of the $k$-torus endowed with Lebesgue measure. Then, we consider
 $X=\text{Spec}(\cW)\times \I_r$ endowed with the product measure $\mu=|\cdot|\times \#(\cdot)$, where 
$\#(\cdot)$ denotes the counting measure on $\I_r$ (as in Remark \ref{rem: sobre medidas}).
We also consider the normalized measure $\tilde \mu=\frac{1}{p\cdot r}\ \mu$ on $X$.
Let $\cF=\{f_j\}_{j\in\In}\in \mathfrak {B}_\alpha(\cW)$ and set $\beta_j(x)=\|\Gamma f_j(x)\|^2$ 
for $x\in \text{Spec}(\cW)$ and $j\in\In\, $. Notice that 
\beq\label{ecua betaj}
 \int_{\text{Spec}(\cW)} \beta_j(x)\ dx=\|f_j\|^2=\alpha_j\ , \peso{for} j\in\In\ .
\eeq
Let $\breve\ga\coma \breve\beta:X\rightarrow \R_+$ be the measurable functions determined by 
$$
\breve \ga (x,j)=\frac{\alpha_j}{p}  \py \breve \beta(x,j)=\beta_j(x) \peso{for} x\in \text{Spec}(\cW) \py j\in\I_{r} \ . 
$$
Consider the map $D:L^\infty(X,\tilde \mu) \rightarrow L^\infty(X,\tilde \mu)$ given by 
$$
D(h)(x,j)=r\cdot\int_{\text{Spec}(\cW)\times \{j\}} h \ d\tilde \mu 
= \frac 1p \ 
\int_{\text{Spec}(\cW)} h(x,j) \ dx \peso{for} x\in \text{Spec}(\cW) \py j \in \I_{r}\ .
$$ 
Then, it is easy to see that $D$ is positive, unital and trace preserving i.e. 
$D$ is a doubly stochastic map; moreover, by Eq. \eqref{ecua betaj}, 
$D(\breve \beta)=\breve \ga$ and by Theorem \ref{teo porque mayo} we conclude that $\breve \ga\prec \breve \beta\,$. 

\pausa
Now, consider the measurable vector-valued function $\beta^\downarrow(x)=(\beta^\downarrow_j(x))_{j\in\In}$ 
obtained by re-arrangement of the entries of the vector $\beta(x)=(\beta_j(x))_{j\in\In}$, for $x\in Z$ independently.  
By construction we get the submajorization relations $(\beta_j(x))_{j\in\I_{r}}\prec_w  (\beta^\downarrow_j(x))_{j\in\I_{r}}$ 
for every $x\in Z$ (notice that we are considering just the first $r$ entries of these $n$-tuples). 

\pausa
Thus, if we consider the measurable function $\breve {\beta^\downarrow} :X\rightarrow \R_+$ 
determined by $\breve{\beta^\downarrow}(x,j)=\beta^\downarrow_j(x)$ if $x\in \text{Spec}(\cW)$ and $j\in\I_{r}\,$, 
then Lemma \ref{lem utilisima} 
shows that 
$\breve\beta \prec_w \breve{\beta ^\downarrow}$ in $(X,\tilde \mu)$.
By transitivity, we conclude that $\breve \ga\prec_w \breve{\beta^\downarrow}$.
By Remark \ref{recordando waterfilling} there exists a unique $b\geq \text{ess-}\inf\limits_{x\in X} \breve{\beta^\downarrow} (x)$ such that the waterfilling of $\breve{\beta^\downarrow}$ at level $b$, denoted $\breve{\beta^\downarrow}_b$, satisfies 
$$
\int_X \breve{\beta^\downarrow} _b \ d\tilde \mu=(r\cdot p)^{-1}\, \sum_{i\in\In} \alpha_i
\geq \int_X \breve{\beta^\downarrow}  \ d\tilde \mu \ .
$$ 
Similarly, let $c\geq \text{ess-}\inf\limits_{x\in X} \,\breve \ga(x)$ be such that the waterfilling of $\breve \ga$ at level $c$, denoted $\breve \ga_c\,$, satisfies 
$$
\int_X \breve\ga_c \ d\tilde \mu=(r\cdot p)^{-1}\, \sum_{i\in\In} \alpha_i\geq \int_X \breve\ga \ d\tilde \mu\ .
$$ 
Therefore, by Lemma \ref{lem: wat er filling}, we see that 
\beq\label{eq relac fc fbetaparaabajo}
\breve\ga_c\prec \breve{\beta^\downarrow} _b\peso{in} (X,\tilde \mu)\ .
\eeq 
By Lemma 
\ref{lem spect represent ese} there exist measurable functions $\lambda_j:\T^k\rightarrow \R_+$ for $j\in\I_{d}$ such that 
we have a representation of $[S_{E(\cF)}]_x=S_{\Gamma \cF(x)}$ as in Eq. \eqref{lem repre espec S}, in terms of some measurable vector fields $v_j:\T^k\rightarrow \ell^2(\Z^k)$ for $j\in\I_d$, such that $\{v_j(x)\}_{j\in\I_d}$ is a ONB of $J_\cW(x)$ for a.e. $x\in \text{Spec}(\cW)$; indeed, in this case $\la_j(x)=0$ for $j\geq r+1$ and a.e. $x\in \text{Spec}(\cW)$.

\pausa
 If we let $e(x)\geq 0$ be determined by the condition
$$ 
\sum_{i\in\I_r}\max\{\beta^\downarrow _i(x),e(x)\}=\sum_{i\in\I_{r}}\lambda_i(x)\ 
\Big(\, =\sum_{i\in\I_{d}}\lambda_i(x)\, \Big) \ , \peso{for a.e.} x\in \text{Spec}(\cW)
$$ 
then by \cite{MR10} (also see \cite{MRS13,MRS14b,MRS14}) we have that 
\beq\label{eq rel mayo MR}
(\delta_i(x))_{i\in\I_{r}}\igdef(\max\{ \beta^\downarrow_i(x),\, e(x)\} )_{i\in\I_{r}}\prec (\lambda_i(x))_{i\in\I_{r}} \ , 
\peso{for a.e.} x\in \text{Spec}(\cW)\,.
\eeq 
Notice that the vector $(\delta_i(x))_{i\in\I_{r}}$ can be considered as the (discrete) waterfilling of the vector $(\beta^\downarrow_j(x))_{j\in\I_{r}}$ at level $e(x)$, for $x\in \text{Spec}(\cW)$.
If $\breve\delta \coma \breve\lambda:X\rightarrow \R_+$ are the measurable functions given by 
$$
\breve\delta(x,j)=\delta_j(x)  \py \breve\lambda(x,j)=\lambda_j(x) \peso{for} x\in \text{Spec}(\cW)  \py j\in\I_{r}  
$$ 
then,  by Lemma \ref{lem utilisima}, we get that
$\breve\delta\prec \breve\lambda$ in $(X,\tilde\mu)$. Notice that by construction, $\breve\delta\geq \breve{\beta^\downarrow} $ and 
$$\int_X \breve\delta\ d\tilde \mu
=(r\cdot p)^{-1}\,\sum_{i\in\In}\alpha_i \,.$$
Hence, by Remark \ref{recordando waterfilling}, we get that $\breve{\beta^\downarrow}_b\prec \breve\delta\,$. 
Putting all the pieces together, we now see that
\beq\label{eq relac major func}
\breve \ga_c\prec \breve{\beta^\downarrow}_b\prec \breve\delta\prec \breve\lambda\ , \peso{in} (X,\tilde\mu)\ .
\eeq
Recall that by construction, we have that 
\beq\label{eq la pinta de fc}
 \breve \ga_c(x)=\max\{ \frac{\alpha_j}{p} \coma c\} \ , \peso{for} x\in \text{Spec}(\cW)\times \{j\}\subset X \, , \ j\in\I_{r}\  . 
\eeq
Then, it is straightforward to check that 
\beq\label{eq. c es el correcto}
(r\cdot p)^{-1}\,\sum_{i\in\In}\alpha_i=\int_X \breve\ga_c\ d\tilde\mu= r^{-1} \cdot 
\sum_{j\in\I_{r}}\max\{ \frac{\alpha_j}{p} \coma c\} \implies 
(\frac{\alpha_j}{p})_{j\in\In}\prec (\max\{\frac{\alpha_j}{p} \coma c\})_{j\in\I_{r}}\ .
\eeq
Thus, by Theorem \ref{teo sobre disenio de marcos},
there exists a Bessel sequence $\cF^{\rm op}=\{f^{\rm op}_i\}_{i\in\In}\in\cW^n$ 
such that the fine spectral structure 
$(\lambda_j([S_{E(\cF^{\rm op})}]_x)\,)_{j\in\N}$ 
satisfies Eq. \eqref{eq defi la op unif dim} and such that 
$\|\Gamma f^{\rm op}_i(x)\|^2=\frac{\alpha_i}{p}\,$, for $i\in\In\,$, and $x\in \text{Spec}(\cW)$.
In particular, $\|f^{\rm op}_i\|^2=\alpha_i$ for $i\in\In\,$, so $\cF^{\rm op}\in \mathfrak {B}_\alpha(\cW)$. If $\varphi\in\convf$ then, by the majorization relations in Eq. \eqref{eq relac major func} 
and  Lemma \ref{lem utilisima},  
\begin{eqnarray*}\label{eq desi poten}
P_\varphi^\cW(E(\cF^{\rm op}))&=&\int_{\text{Spec}(\cW)} [ \sum_{j\in\I_{r}}\varphi(\max\{\frac{\alpha_j}{p} \coma c\})
+ (d-r)\, \varphi(0) ]\ dx= \int_X \varphi\circ \breve \ga_c\ d\mu +  p\,(d-r)\, \varphi(0) \\ 
&\leq & \int_X \varphi\circ \breve\lambda\ d\mu   + p\,(d-r)\, \varphi(0) 
=P_\varphi^\cW(E(\cF))\,.
\end{eqnarray*} 
Hence, $\cF^{\rm op}$ satisfies items 1. and 2. in the statement.
\end{proof}

\pausa
The previous result shows that there are indeed structural optimal frames with prescribed norms in the sense that these frames minimize any frame potential within $\mathfrak {B}_\alpha(\cW)$; along its proof we showed several majorization relations that allow us to prove that the spectral structure of any such structural optimal frame is described by Eq. \eqref{eq defi la op unif dim}.

\begin{teo}[Fine spectral structure of optimal sequences in $\mathfrak {B}_\alpha(\cW)$] 
\label{teo struct fina dim hom}
With the hypothesis and notations from Theorem \ref{teo dim unif},
assume that $\cF\in \mathfrak {B}_\alpha(\cW)$ is such that there exists $\varphi\in\convfs$ with $P_\varphi^\cW(E(\cF))=P_\varphi^\cW(E(\cF^{\rm op}))$.
Then, for a.e. $x\in \text{Spec}(\cW)$ we have that  
\beq\label{eq defi la op unif dim2}
\la_j([S_{E(\cF)}]_x)  =\left\{
  \begin{array}{ccc}
     \max\{\frac{\alpha_j}{p} \coma c\} = \max\{ \beta^\downarrow_j(x)\, ,\, c\}& if & j\in\I_{r} \ ; \\
      0 & if &  r+1 \le j\le d\  ,
   \end{array}
	\right.
	\eeq
where $\beta^\downarrow_1(x)\geq \ldots\beta^\downarrow_n(x)\geq 0$ are obtained by re-arranging the sequence  
$$
\beta(x) = \big( \, \beta_1(x) \coma \ldots \coma \beta_n(x)\,\big) 
=\big(\, \|\Gamma f_1(x)\|^2 \coma \ldots \coma \|\Gamma f_n(x)\|^2 \,\big) \in \R^n 
$$ 
in non-increasing order, independently for each $x\in\text{Spec}(\cW)$.
\end{teo}
\begin{proof}
We continue to use the notations and terminology from the proof of Theorem \ref{teo dim unif}.
Assume further that $\cF\in \mathfrak {B}_\alpha(\cW)$ is such that there exists $\varphi\in\convfs$ with 
$$
p\cdot \sum_{j\in\I_{r}} \varphi (\max\{\frac{\alpha_j}{p}\, , \, c\}) + 
p\,(d-r)\,\varphi(0) = P_\varphi^\cW(E(\cF))\ .
$$ 
Then, using this last fact and Lemma \ref{lem utilisima} we see that 
$$
(r\cdot p)\,\int_X \varphi\circ \breve \ga_c\  d\tilde\mu=
(r\cdot p)\,\int_X \varphi\circ \breve\lambda\ d\tilde\mu\ .
$$
Hence, by Eq. \eqref{eq relac major func} we have that
$$
\int_X \varphi\circ \breve \ga_c\  d\tilde\mu= \int_X \varphi\circ \breve{\beta^\downarrow }_b\  d\tilde\mu
= \int_X \varphi\circ \breve\delta\  d\tilde\mu= \int_X \varphi\circ \breve\lambda\  d\tilde\mu\ .
$$ 
Thus, by Proposition \ref{pro int y reo} the functions $\breve \ga_c,\,\breve{\beta^\downarrow}_b,\,\breve\delta,\,\breve\lambda$ 
are equimeasurable. On the one hand, Eq. \eqref{eq rel mayo MR}, together with the equality above imply that 
$\max\{ \beta^\downarrow_j(x)\coma e(x)\} =\lambda_j(x)$, for $j\in\I_{r}$ and a.e. $x\in \text{Spec}(\cW)$ and hence, 
by construction, $\breve\delta=\breve\lambda\,$.
On the other hand, by \cite[Corollary 5.6]{BMS15} we also get that $\breve{\beta^\downarrow} _b=\breve\delta\,$. 
Therefore, $\breve{\beta^\downarrow} _b=\breve\delta=\breve\lambda\,$; in particular, we get that 
$\max\{ \beta^\downarrow_j(x)\coma b\} =\lambda_j(x)$, for $j\in\I_{r}$ and a.e. $x\in \text{Spec}(\cW)$. 

\pausa
Notice that, since $\breve \ga_c$ and $\breve\lambda$ are equi-measurable, then 
$|\breve\lambda^{-1}(\max\{\frac{\alpha_j}{p}\coma c\})|=|{\breve \ga_c} ^{-1}(\max\{\frac{\alpha_j}{p}  \coma c\})|$ for $j\in\I_{r}\,$; 
thus, $\breve\lambda$ takes the values $\max\{\frac{\alpha_j}{p},\,c\}$ for $j\in\I_{r}$ (off a zero-measure set). 
As $\breve\lambda$ and $\breve \ga_c$ are both induced by the vector-valued functions 
$$
\text{Spec}(\cW)\ni x\mapsto (\max\{\frac{\alpha_j}{p} \coma c\})_{j\in\I_{r}}\in(\R_+^r)^\downarrow
\peso{and} \text{Spec}(\cW)\ni x\mapsto (\lambda_j(x))_{j\in\I_{r}}\in(\R_+^r)^\downarrow
$$ 
respectively, we conclude that 
$$ 
(\max\{\frac{\alpha_j}{p}\coma c\})_{j\in\I_{r}}=(\lambda_j(x))_{j\in\I_{r}}
=(\max\{ \beta^\downarrow_j(x)\coma b\} )_{j\in\I_{r}} \ , 
\peso{for} x\in \text{Spec}(\cW)\ .
$$  
From this last fact, we see that we can set $b=c$ and the result follows.
\end{proof}

\begin{rem}\label{interpret de prop dim unif} Consider the notations and terminology from Theorem 
\ref{teo dim unif}. We point that there is a simple 
formula for the constant $c$. Indeed, notice that if 
$\cF^{\rm op}\in \mathfrak {B}_\alpha(\cW)$ is the structural solution of the optimization problem considered in Theorem
\ref{teo dim unif} then 
$$
\sum_{j\in\I_r}\lambda_j([S_{E(\cF^{\rm op})}]_ x)
	=\tr([S_{E(\cF^{\rm op})}]_ x)=\sum_{j\in\In}\|\Gamma f^{\rm op}_j(x)\|^2\peso{for a.e.} x\in\T^k 
$$
Therefore, 
\beq\label{formu c}
 \sum_{i\in\I_r}\max\{\frac{\alpha_i}{p} \coma c\}=\frac 1p \ \sum_{j\in\In}\alpha_j\ ,
\eeq
which shows that $c$ is obtained by the previous discrete waterfilling condition. \EOE
\end{rem}

\pausa  Tight frames play a central role in applications. On the one hand, they give raise to simple
reconstruction formulas; on the other hand, they have several robustness 
properties related with numerical stability of the encoding-decoding scheme
that they induce. It is therefore 
important to have conditions that assure the existence of tight frames with prescribed norms: in the finite dimensional context 
(i.e. finite frame theory) this problem is solved in \cite{CKFT} in terms of the so-called fundamental inequality. As a consequence of Remark 
\ref{interpret de prop dim unif}, we obtain conditions for the existence of tight 
SG frames with norms given by a finite sequence of positive numbers, in the uniform dimensional case.

\begin{cor} \label{coro tight 2}
Consider the notations and hypothesis of Theorem \ref {teo dim unif}. 
In the uniform dimensional case (so in particular, $d(x)=d$ for a.e. $x\in \text{Spec}(\cW)\,$), we have that 
$$
\text{\rm there exist {\bf tight} frames in $\mathfrak {B}_\alpha(\cW)$ } \  \iff  \ \ 
d=r\le n \peso{ \rm and}  d \cdot\al_1  \le \sum_{j\in\In}\alpha_j\ .
$$
\end{cor}
\proof It is a direct consequence of Eqs. \eqref{eq defi la op unif dim} and \eqref{formu c}.
\qed
\subsection{Existence and structure of $P^\cW_\varphi$-minimizers in $\mathfrak {B}_\alpha(\cW)$: the general case}\label{subsec gral mi gral}

It turns out that Theorem \ref{teo dim unif} allows to reduce the study of the 
spectral structure of minimizers of convex potentials in FSI subspaces with norm restrictions to a finite dimensional model. 
Indeed, consider the Notations \ref{nota impor2} and, for the sake of simplicity, assume that $p_i>0$ for every 
$i\in\I_{\ell}\,$. 
Consider $\alpha\in (\R_{>0}^n)^\downarrow$ and let $\cF\in \mathfrak {B}_\alpha(\cW)$. 
For each $i\in\I_\ell$ let  $\cW_i\subset L^2(\R^k)$ be the closed FSI subspace whose 
fibers coincide with those of $\cW$ in $Z_i=d^{-1}(i)$ and are the zero subspace elsewhere, 
and let $\cF_i=\{f_{i,j}\}_{j\in\In}\in\cW_i^n$ be determined by 
$$
\Gamma f_{i,j}(x)=\chi_{Z_i}(x)\ \Gamma f_j(x) \peso{for a.e.} x\in\T^k \py j\in\In\ ,
$$
where $\chi_Z$ denotes the characteristic function of a measurable set $Z\subset\T^k$. 
Fix a convex function $\varphi\in\convf$. Since each $\cW_i$ is also a uniform FSI, it satisfies the 
hypothesis of Theorem \ref{teo dim unif}. Then we conclude that for each $i\in\I_\ell$ there exists $\cF_i^{\rm dis}=\{f_{i,j}^{\rm dis}\}_{j\in\In}\in\cW_i ^n$ such that 
$$
\|f_{i,j}^{\rm dis}\|^2=\|f_{i,j}\|^2 \peso{for} j\in\I_n \py P^{\cW_i}_\varphi (E(\cF_i^{\rm dis}))\leq P^{\cW_i}_\varphi (E(\cF_i)) \peso{for} i\in\I_\ell\ .
$$ 
We can recover the initial family $\cF=\{f_i\}_{i\in\In}$ by gluing together the families $\cF_i$ for $i\in\I_\ell\,$. 
Similarly, if we glue the families $\cF_i^{\rm dis}$ we get a family 
$\cF^{\rm dis}$ (in such a way that $(\cF^{\rm dis})_i=\cF_i^{\rm dis}\in \cW_i^n$ as before, for $i\in\I_\ell$). 
Notice that $\cF^{\rm dis}\in \mathfrak {B}_\alpha(\cW)$ since 
$$ 
\|f_i^{\rm dis}\|^2=\sum_{j\in\I_n}\|f_{i,j}^{\rm dis}\|^2= \|f_i\|^2=\alpha_i \peso{for} i\in\I_n \ ,
$$ 
using the fact that the subspaces $\{\cW_i\}_{i\in\I_\ell}$ are mutually orthogonal. Also  
$$
P^{\cW}_\varphi (E(\cF^{\rm dis}))= \sum_{i\in\I_\ell} P^{\cW_i}_\varphi (E(\cF_i^{\rm dis}))\leq   
\sum_{i\in\I_\ell} P^{\cW_i}_\varphi (E(\cF_i))= P^{\cW}_\varphi (E(\cF))\ .
$$ 
Now, the fine spectral structure of $\cF_i^{\rm dis}$
is of a discrete nature (as described in Theorem \ref{teo dim unif}). Moreover, this fine structure is 
explicitly determined in terms of the matrix 
\beq\label{las B}
B=(p_i^{-1}\, \|f_{i,j}\|^2)_{i\in \I_\ell,\,j\in\In} \in \R_{+}^{\ell\times n}  \peso{fulfilling the identity}
p^T\,B=\alpha  \ , 
\eeq
where $p=(p_i)_{i\in\I_\ell}$ and $\alpha=(\alpha_i)_{i\in\In}\,$. 
Notice that the set of all such matrices form a convex compact subset of $\R_{+}^{m\times n}$.
The advantage of this approach is that we can use simple tools such as convexity, 
compactness and continuity in a finite dimensional context, to show existence of 
optimal spectral structure within our reduced model. Nevertheless, the reduced 
model has a rather combinatorial nature (see the definition of 
$\Lambda_{\alpha,\,p}^{\rm op}(\delta)$ below), so we build it in steps.

\begin{notas}\label{muchas nots} In order to simplify the exposition of the next result, 
we introduce the following notations that are motivated by the remarks above. Let  $ m \coma n\in\N$:
\ben
\item Inspired in Eq. \eqref{las B}, for finite sequences $\alpha\in (\R_{>0}^n)^\downarrow$ 
and $p=(p_i)_{i\in\I_{m}}\in \R_{>0}^m$ we consider the set of weighted partitions 
$$
\barr{rl}
W_{\alpha,\,p} 
&=\{ B\in \R_+^{m\times n }  \ : \ p^T\, B = \al \, \} 
\ . \earr
$$ 
It is straightforward to check that $W_{\alpha,\,p}$ is a convex compact set. 
\item \label{item2} 
Given $d\in \N$ we define the map $L_d : \R_+^n \to (\R_+^d)^\downarrow$ given by 
\beq\label{eq defi gammacd}
L_d (\gamma ) 
\igdef  \left\{
  \begin{array}{ccc}
    (\max\{\gamma\da_i\coma c_d(\gamma) \})_{i\in\I_{d}}  & 
    & \text{if }  \ d\leq n   \\
      (\gamma\da,0_{d-n}) &  &  \text{if }  \ d>  n  
   \end{array}  \peso{for every} \gamma \in \R_+^n \ ,
	\right.
\eeq
where the constant $c_d(\gamma) \in \R_+$ is uniquely determined by $\tr \, L_d (\gamma ) = \tr \, \gamma$, in case $d\le n$. 
By \cite[Prop. 2.3]{MR10} we know that $\gamma\prec L_d (\gamma )\,$,  and $L_d (\gamma )\prec \beta$
for every $\beta\in\R^d$ such that $\gamma\prec \beta$. 
\item\label{item3} Let $\delta=(d_i)_{i\in\I_{m}}\in\N^m$ be such that $1 \le d_1< \ldots< d_m$. For each 
$B \in W_{\alpha,\,p}$ consider 
\beq\label{Bdelta}
B_{\delta}= \big[ \, L_{d_i} (R_i(B)\,)\, \big]_{i\in\I_{m}} 
\in \prod_{i\in\I_{m}} (\R_+^{d_i})^\downarrow 
\ ,
\eeq
where $R_i(B)\in \R_+^n$ denotes the $i$-th row of $B$. 
Moreover, using the previous notations we introduce the {\it reduced model (for optimal spectra)} 
$$ 
\Lambda^{\rm op}_{\alpha,\,p}(\delta)\igdef \{ B_\delta :\ B\in W_{\alpha,\,p}\}
\subset \prod_{i\in\I_{m}} (\R_+^{d_i})^\downarrow\, .
$$
In general, $\Lambda^{\rm op}_{\alpha,\,p}(\delta)$ is not a convex set and indeed, 
the structure of this set seems rather involved; notice that 
item 2 above shows that the elements of $\Lambda^{\rm op}_{\alpha,\,p}(\delta)$ 
are $\prec$-minimizers within appropriate sets. 
\EOE
\een
\end{notas}
\pausa
The following result describes the existence and uniqueness of the solution to an optimization 
problem in the reduced model for a fixed $\varphi\in\convfs$, which corresponds to the minimization of the convex potential $P^\cW_\varphi$ in $\mathfrak {B}_\alpha(\cW)$ for a 
FSI subspace $\cW$ and a sequence of weights $\alpha\in (\R_{>0}^n)^\downarrow$. The proof of this result is presented in section \ref{subsec reduced} (Appendix).

\begin{teo}\label{teo estruc prob reducido unificado}
Let $ m,\, n\in\N$, $\alpha\in (\R_{>0}^n)^\downarrow$, $p=(p_i)_{i\in\I_{m}}\in \R_{>0}^m$ and 
$\delta=(d_i)_{i\in\I_{m}}\in\N^m$ be such that $1\leq d_1< \ldots< d_m$.
If $\varphi\in\convf$ then
there exists $\Psi^{\rm op}=[\psi_i^{\rm op}]_{i\in\I_{m}}\in \Lambda^{\rm op}_{\alpha,\,p}(\delta)$ such that 
$$  \sum_{i\in\I_{m}} {p_i}\,\tr(\varphi(\psi_i^{\rm op}) )  
\leq \sum_{i\in\I_{m}} {p_i}\,\tr(\varphi(\psi_i) )
 \peso{for every} \Psi=[\psi_i]_{i\in\I_{m}}\in \Lambda^{\rm op}_{\alpha,\,p}(\delta)\,.$$
Moreover:
\begin{enumerate}
\item If $\varphi\in\convfs$ then such $\Psi^{\rm op}$ is unique;
\item If $n\geq d_m$ and $\varphi\in\convfs$ is differentiable in 
$\R_+\,$ then $\Psi^{\rm op}\in \prod_{i\in\I_m} (\R_{>0}^{d_i})^\downarrow$.
\qed
\end{enumerate}
\end{teo}

\pausa
We now turn to the statement and proof of our main result in this section (Theorem \ref{teo min pot fsi generales} below). Hence, we let $\cW$ be an arbitrary FSI subspace of $L^2(\R^k)$ and let $\alpha=(\alpha_i)_{i\in\In}\in (\R_{>0}^n)\da$. Recall that 
$$
\mathfrak {B}_\alpha(\cW)=\{\cF=\{f_i\}_{i\in\In}\in\cW^n:\ E(\cF) \ \text{is a Bessel sequence }, \ \|f_i\|^2=\alpha_i\,,\ i\in\In\}\,.
$$ 
Given $\varphi\in\convf$, in what follows we show the existence of finite sequences $\cF^{\rm op}\in \mathfrak {B}_\alpha(\cW)$ such that 
$$
P_\varphi^\cW(E(\cF^{\rm op}))=\min\{P_\varphi^\cW(E(\cF)): \ \cF\in \mathfrak {B}_\alpha(\cW)\}\,.
$$ Moreover, in case $\varphi\in\convfs$ then we describe the fine spectral structure of the frame operator of $E(\cF^{\rm op})$ of any such $\cF^{\rm op}$.
%

\begin{teo}\label{teo min pot fsi generales}
 Let $\alpha=(\alpha_i)_{i\in\In}\in (\R_{>0}^n)\da$, 
consider the Notations \ref{nota impor2} 
and fix $\varphi\in\convf$. 
Then, there exists $\cF^{\rm op}\in \mathfrak {B}_\alpha(\cW)$ such that:
\ben
\item $\la_j([S_{E(\cF^{\rm op})}]_x)=:\psi^{\rm op}_{i,j}\in\R_+$ is a.e. constant for $x\in Z_i$, $j\in \I_{i}$ and $i\in\I_\ell$;
\item For every $\cF\in \mathfrak {B}_\alpha(\cW)$ we have that 
$$
\sum_{i\in\I_{\ell}} {p_i}\left( \sum_{j\in\I_i}\varphi(\psi^{\rm op}_{i,j})\right)=P_\varphi^{\cW}(E(\cF^{\rm op}))\leq P_\varphi^{\cW}(E(\cF))\ .
$$ 
\een
 If we assume that $\varphi\in\convfs$ then:
\ben
\item[a)] If $\cF\in \mathfrak {B}_\alpha(\cW)$ is such that $P_\varphi^{\cW}(E(\cF))=P_\varphi^{\cW}(E(\cF^{\rm op}))$ then $S_{E(\cF)}$ has the same fine spectral structure as $S_{E(\cF^{\rm op})}$. 
\item[b)] If we assume further that $\varphi$ is differentiable in $\R_+$ and that 
$n\geq i$ for every $i\in\I_\ell$ such that $p_i=|Z_i|>0$, then $E(\cF)$ is a frame for $\cW$.
\een
\end{teo}
\begin{proof}
Without loss of generality, we can assume that there exists an $m\leq \ell $ such that $p_i=|Z_i|>0$ for $i\in\I_{m}$ and $p_i=|Z_i|=0$ for $m+1\leq i\leq \ell$
(indeed, the general case follows by restricting the argument given below to the set of indexes $i\in\I_\ell$ for which $p_i=|Z_i|>0$).
We set $p=(p_i)_{i\in\I_m}\in\R^m_{>0}$ and consider
$\cF=\{f_i\}_{i\in\In}\in \mathfrak {B}_\alpha(\cW)$. For $i\in\I_{m}$ and $j\in\In$ set 
$$
B_{i,j}\igdef\frac{1}{ p_i}\,\int_{ Z_i} \|\Gamma f_j(x)\|^2\ dx
\implies \sum_{i\in\I_{m}}  p_i\,B_{i,j}
=\int_{\text{Spec}(\cW)}\|\Gamma f_j(x)\|^2\ dx=\|f_j\|^2=\alpha_j \ ,
$$
for every $j\in\In\,$, since $\text{Spec}(\cW)=\cup_{i\in\I_m}Z_i$.
Then $p^T\, B=\alpha $ so using Notations \ref{muchas nots}, $B\in W_{\alpha,\, p}\,$.

\pausa Now, fix $i\in\I_{m}$ and consider the weights
$\beta^i= p_i\, R_{i}(B)^\downarrow\in\R_+^n\,$. For the sake of simplicity we assume, 
without loss of generality, that $\beta^i=p_i\, R_{i}(B)$.  For $i\in\I_m\,$, 
let $\cW_i$  be the FSI subspace whose fibers coincide with those of $\cW$ inside $Z_i$ and that 
are the zero subspace elsewhere; hence, Spec$(\cW_i)= Z_i$ and $\dim J_{\cW_i}(x)=i$ for $x\in\text{Spec}(\cW_i)$. 
For $i\in\I_m\,$, 
set $\cF_i=\{f_{i,j}\}_{j\in\In}$ where $\Gamma f_{i,j}(x)=\Gamma f_{j}(x)$ for $x\in Z_i$ and $\Gamma f_{i,j}(x)=0$ elsewhere; then $\cF_i\in \mathfrak B_{\beta^i}(\cW_i)$ and  
$$ 
[S_{E(\cF_i)}]_x=S_{\Gamma \cF_i(x)}=[S_{E(\cF)}]_x \peso{for} x\in Z_i=\text{Spec}(\cW_i) \ , \quad i\in\I_m\, .$$
If we consider the minimization of $P_\varphi^{\cW_i}$ in $\mathfrak B_{\beta^i}(\cW_i)$ then, Theorem \ref{teo dim unif} and Remark \ref{interpret de prop dim unif} imply that there exists $c_i\geq 0$ 
such that
\beq \label{desi caso uniforme}
  p_i\, \sum_{j\in\I_{i}} \varphi (\max\{B_{i,j}\, , \, c_i\}) \leq P_\varphi^{\cW_i}(E(\cF_i))
\py \sum_{j\in I_{i}} \max\{B_{i,j}\, , \, c_i\}=\sum_{i\in\In} B_{i,j} 
\,.
\eeq
Using Notations \ref{muchas nots} and Eq. \eqref{eq. c es el correcto}, we get that for $i\in\I_m$
$$
L_{i}(R_{i}(B))= (\max\{B_{i,j}\, , \, c_i\})_{j\in I_{i}}
\implies B_{\delta}=[(\max\{B_{i,j}\, , \, c_i\})_{j\in \I_{i}}]_{i\in\I_{m}}\in \Lambda^{\rm op}_{\alpha,\,p}(\delta)\,,$$
 where $\delta=(i)_{i\in\I_m}$. Notice that $\cW=\oplus_{i\in\I_{m}}\cW_i$ (orthogonal sum) and hence 
$$ 
\sum_{i\in\I_{m}}  p_i\, 
\sum_{j\in\I_{i}} \varphi (\max\{B_{i,j}\, , \, c_i\}) 
\le  \sum_{i\in\I_{m}} P_\varphi^{\cW_i}(E(\cF_i))=P_\varphi^\cW(E(\cF))\,. $$

\pausa 
Let $[\psi^{\rm op}_i]_{i\in\I_{m}}=\Psi^{\rm op}
\,\in \Lambda_{\alpha,\,p}^{\rm op}(\delta)$ be as in 
Theorem \ref{teo estruc prob reducido unificado}. Then
\beq\label{desi potop} 
\sum_{i\in\I_{m}} { p_i}\, \tr(\varphi(\psi^{\rm op}_i))
= \sum_{i\in\I_{m}} {p_i}\left( \sum_{j\in\I_i}\varphi(\psi^{\rm op}_{i,j})\right)
\leq \sum_{i\in\I_{m}}  p_i\, \sum_{j\in\I_{i}} 
\varphi (\max\{B_{i,j}\, , \, c_i\})\leq P_\varphi^\cW(E(\cF))\ .
\eeq
Recall that by construction, there exists 
$B^{\rm op}=(\gamma_{i,j})_{(i\coma j)\in\I_{m}\times\In}\in W_{\alpha,\,p}$ 
such that $B^{\rm op}_{\delta}=\Psi^{\rm op}$ (see item \ref{item3} in Notations \ref{muchas nots}). In this case, 
$$
\psi^{\rm op}_i=L_{i}(\,(\gamma_{i,j})_{j\in\In}) \implies  (\gamma_{i,j})_{j\in\In}\prec \psi^{\rm op}_i \peso{for} i\in\I_{m}\ .
$$
Let $\gamma:\text{Spec}(\cW)\rightarrow \R^n$ be given by $\gamma(x)
=R_i(B^{\rm op})= (\gamma_{i,j})_{j\in\In}$ if $x\in  Z_i$, for $i\in\I_{m}\,$; similarly, let 
$\lambda:\text{Spec}(\cW)\rightarrow \coprod_{i\in\I_{m}}\R^{i}$, $\lambda(x)=\psi^{\rm op}_i$ 
if $x\in  Z_i$, for $i\in\I_{m}\,$. Then, by the previous remarks we get that $\gamma(x)\prec \lambda(x)$ 
for $x\in\text{Spec}(\cW)$.

\pausa
Hence, by Theorem \ref{teo sobre disenio de marcos} there exists 
$\cF^{\rm op}=\{f_j^{\rm op}\}_{j\in\In}$ such that 
$$
\|\Gamma f_j^{\rm op}(x)\|^2=\gamma_{i,j} \py  
\lambda_j([S_{E(\cF^{\rm op})}]_x)=\psi_{i\coma j}^{\rm op} 
\peso{for} x\in  Z_i\, , \ \ j\in\I_i  \peso{and} i\in\I_{m}\ .
$$ 
Since $B^{\rm op}\in W_{\alpha,\,p}$ then
$$\|f_j^{\rm op}\|^2=\int_{\text{Spec}(\cW)}\|\Gamma f_j^{\rm op}(x)\|^2\ dx=\sum_{i\in\I_{m}} p_i \ \gamma_{i,j}=\alpha_j \peso{for} j\in\In\implies \cF^{\rm op}\in 
\mathfrak B_{\alpha}(\cW)$$ and 
\beq\label{eq casi estamos}
P_\varphi^{\cW}(E(\cF^{\rm op}))=\int_{\text{Spec}(\cW)} \tr(\varphi(\lambda(x)))\ dx=\sum_{i\in\I_{m}}  p_i \ \tr(\varphi(\psi_i^{\rm op})),
\eeq
then by Eq. \eqref{desi potop} we see that 
$P_\varphi^{\cW}(E(\cF^{\rm op}))\leq P_\varphi^{\cW}(E(\cF))\,.$
Since $\cF\in \mathfrak B_{\alpha}(\cW)$ was arbitrary, the previous facts show that $\cF^{\rm op}$ satisfies items 1. and 2. in the statement.

\pausa
Assume further that $\varphi\in\convfs$ and $\cF\in \mathfrak {B}_\alpha(\cW)$ is such that $P_\varphi^{\cW}(E(\cF))=P_\varphi^{\cW}(E(\cF^{\rm op}))$.
Then, by Eqs. \eqref{desi caso uniforme}, \eqref{desi potop} and \eqref{eq casi estamos} we see that 
$$  
 p_i\, \sum_{j\in\I_{i}} \varphi (\max\{B_{i,j}\, , \, c_i\}) = P_\varphi^{\cW_i}(E(\cF_i)) \peso{for} i\in\I_{m}\ .
$$
Therefore, by the case of equality in Theorem \ref{teo struct fina dim hom} and the uniqueness of $\Psi^{\rm op}$ from Theorem \ref{teo estruc prob reducido unificado}
we conclude that 
$$ 
\lambda_j([S_{E(\cF)}]_x) =\lambda_j([S_{E(\cF_i)}]_x)=\psi^{\rm op}_{i,\,j} \peso{for} x\in Z_i\, , \ j\in\I_{i}\, , \ i\in\I_{m}\ .
$$
Finally, in case $\varphi\in\convfs$ is differentiable in $\R_+$ and $n\geq m$ then, again by 
Theorem \ref{teo estruc prob reducido unificado}, we see that $S_{E(\cF)}$ is bounded from below in $\cW$ (since the vectors in $\Psi^{\rm op}$ have no zero entries) and hence $E(\cF)$ is a frame for $\cW$. 
\end{proof}

\pausa
We end this section with the following remarks. With the notations of Theorem \ref{teo min pot fsi generales}, notice that the optimal Bessel sequence $\cF^{\rm op}\in \mathfrak {B}_\alpha(\cW)$ depends on the convex function $\varphi\in\convf$, which was fixed in advance. That is, unlike the uniform case, we are not able to show that there exists $\cF^{\rm univ}\in \mathfrak {B}_\alpha(\cW)$ such that $\cF^{\rm univ}$ is a $P^\cW_\varphi$-minimizer in 
$\mathfrak {B}_\alpha(\cW)$ for every $\varphi\in\convf$. It is natural to wonder whether there exists such a universal solution $\cF^{\rm univ}\in \mathfrak {B}_\alpha(\cW)$; we conjecture that this is always the case.

\section{Appendix}\label{Appendixity}

\subsection{The Schur-Horn theorem for measurable fields of self-adjoint matrices and applications}

The simple notion of majorization between real vectors has played an important role in finite frame theory in finite dimensions. 
In particular, it is well known that the existence of finite sequences with prescribed norms and frame operator can be characterized in terms of majorization, applying the Schur-Horn theorem.

\pausa
Next we develop a Schur-Horn type theorem for measurable fields of self-adjoint matrices and use this result to prove Theorem \ref{teo:mayo equiv}. 
Our proof is an adaptation of that given in \cite{HJ13} for the classical Schur-Horn theorem. We will use the existence of measurable eigenvalues and eigenvectors (i.e. diagonalization by measurable fields of unitary matrices) of measurable fields of self-adjoint matrices from \cite{RS95}.
In what follows we consider a measure subspace $(X,\, \mathcal X,\, |\,\cdot|)$ of the measure space $(\T^k,\,\mathcal B(\T^k),\,|\,\cdot|)$ of the $k$-torus with Lebesgue measure on Borel sets.

\begin{teo} \label{teo:mayo y el unitario} Let $A(\cdot): X \to \cH(n)$ be a measurable field of self-adjoint matrices with associated measurable eigenvalues $b_j:X\to \R$ for $j\in \I_n$ such that $b_1\geq \cdots \geq b_n\,$. Let $c_j:X\to \R$ be measurable functions for $j\in \I_n\,$. The following statements are equivalent:
\ben
\item $c(x)=(c_1(x)\coma  \cdots \coma  c_n(x))\prec b(x)=(b_1(x)\coma  \cdots\coma  b_n(x))$,  for a.e. $x\in X$. 
\item There exists a measurable  field of unitary matrices $U(\cdot):X\to \cU(n)$, such that 
\beq\label{eq SH}
d(U(x)^*\,A(x)\ U(x))=c(x)\, , \peso{for a.e.} x\in X\,,
\eeq
where $d(B)\in \C^n$ denotes the main diagonal of the matrix $B\in \cM_n(\C)$.
\een
\end{teo}
\begin{proof}
First notice that the implication $2.\implies 1.$ follows from the classical Schur theorem. 

\pausa
$1.\implies 2.\,$: 
By considering a convenient measurable field of permutation matrices, we can (and will) assume that the entries of the vector $c(x)$ is also arranged in non-increasing order: $c_1(x)\geq c_2(x)\geq \ldots c_n(x)$. By the results from \cite{RS95} showing the existence of a measurable field of unitary matrices diagonalizing the field $A$, we can assume without loss of generality that $A(x)=D_{b(x)}$ where $D_{b(x)}$ is the diagonal matrix with main diagonal $(b_1(x)\coma \ldots \coma b_n(x))$ for a.e. $x\in X$.
 
\pausa We will argue by induction on $n$. For $n=1$ the result is trivial.  
 Hence, we may assume that $n\geq 2$.
Since $c(x)\prec b(x)$, we have $b_1(x)\geq c_1(x)\geq c_n(x)\geq b_n(x)$, so if $b_1(x)=b_n(x)$ it follows that 
all the entries of $c(x)$ and $b(x)$ coincide, $A(x)=c_1(x) I_n\,$, and we can take $U(x)=I_n\,$ for every such $x \in X$. 
By considering a convenient partition of $X$ we may therefore assume that $b_1(x)>b_n(x)$ in $X$.
Similarly, in case $c_1(x)=c_n(x)$ then the unitary matrix $U(x)=n^{-1/2}\, (w^{j\,k})_{j,k\in\In}\,$, 
where $w=e^{\frac{-2\pi i}{ \ n}}$, satisfies that 
$U(x)^*\, D_{b(x)}\ U(x)=(c_1(x)\coma \ldots \coma c_n(x))$. 
Therefore, 
by considering a convenient partition of $X$ we may therefore assume that $c_1(x)>c_n(x)$ in $X$

\pausa
For $n=2$, we have $b_1(x)>b_2(x)$ and $b_1(x)\geq c_1(x)\geq c_2(x)= (b_1(x)- c_1(x))+b_2(x)\geq b_2(x).$ Consider the matrix
 $$U(x)=\frac{1}{\sqrt{b_1(x)-b_2(x)}}
\begin{pmatrix} \sqrt{b_1(x)-c_2(x)} &-\sqrt{c_2(x)-b_2(x)} \\
\sqrt{b_2(x)-c_2(x)} &\sqrt{b_1(x)-c_2(x)} 
\end{pmatrix} \peso{for a.e.} x\in X\,.  
$$
Notice that $U(x):X \to M_2(\C)^+$ is a measurable function and an easy computation reveals that $U(x)^*\,  U(x)=I_2$, so $U(x)$ is unitary for a.e. 
$x\in X$. A further computation shows that 
$$U(x)^*\, A(x)\  U(x)=
\begin{pmatrix} &c_1(x)& &*& \\
&*& &c_2(x)& 
\end{pmatrix} \peso{for a.e.} x\in X\,.  
$$
That is, $d(U^*(x)\, A(x)\, U(x))=(c_1(x),\,\ c_2(x))$ and $U(\cdot)$ has the desired properties.

\pausa 
Suppose that $n\geq 3$ and asssume that the theorem is true if the vectors $c(x)$ and $b(x)$ have size at most $n-1$. 
For each $x\in X$ let $k(x)$ be the largest integer $k\in\In$ such that $b_k(x)\geq c_1(x)$. Since $b_1(x)\geq c_1(x)>c_n(x)\geq b_n(x)$, we see that $1\leq k\leq n-1$. Then, by considering a convenient partition of $X$ into measurable sets we can assume that $k(x)=k$ for $x\in X$. Therefore, by definition of $k$ we get that $b_k(x)\geq c_1(x)>b_{k+1}(x)$ for $x\in X$. 
Let $\eta(x)=b_k(x)+b_{k+1}(x)-c_1(x)$ and observe that $\eta(x)=(b_k(x)-c_1(x))+b_{k+1}(x)\geq b_{k+1}(x)$. Then, the measurable vector $(b_k(x), b_{k+1}(x))$ majorizes the measurable vector $(c_1(x), \eta(x))$ and $b_k(x)>b_{k+1}(x)$ for a.e. $x\in X$.
 Let $$D_1(x)=\begin{pmatrix} &b_k(x)& &0& \\
&0& &b_{k+1}(x)& 
\end{pmatrix} \peso{for a.e.} x\in X\,.$$
By the case $n=2$ we obtain a measurable field of unitary matrices $U_1(\cdot):X\rightarrow \cU(2)$ such that 
$$d(U_1(x)^*\, D_1(x)\, U_1(x))= (c_1(x), \eta(x)) \peso{for a.e.} x\in X\,.$$ 
Since $b_k(x)=\eta(x)+(c_1(x)-b_{k+1}(x))>\eta(x)$, we have:

\pausa
If $k=1$ then $b_1(x)>\eta(x)\geq b_2(x)\geq \cdots \geq b_n(x)$; if we let $D_2(x)\in \M_{n-2}(\C)$ be the diagonal matrix with main diagonal $(b_3(x)\coma  \ldots\coma  b_n(x))$ then  $D_{b(x)}=D_1(x)\oplus D_2(x)$ and
$$\begin{pmatrix} U_1(x)& 0 \\
0 &I_{n-2}
\end{pmatrix}^*
\begin{pmatrix} D_1(x)& 0 \\
0 &D_2(x) 
\end{pmatrix}
\begin{pmatrix} U_1(x) &0 \\
0 &I_{n-2}
\end{pmatrix}=\begin{pmatrix}c_1(x) &Z(x)^*\\
Z(x) &V_1(x)\end{pmatrix}$$
where $Z(x)^*=(\overline{z(x)}\coma  0\coma  \ldots\coma  0)\in M_{1,(n-1)}(\C)$, $z(\cdot):X\rightarrow \C$ is a measurable function and $V_1(x)\in\M_{n-1}(\C)$ 
is the diagonal matrix with main diagonal $(\eta(x)\coma b_3(x)\coma \ldots\coma b_n(x))$.
Moreover, in this case it turns out that $(\eta(x)\coma b_3(x)\coma \cdots \coma b_n(x))$ 
majorizes $(c_2(x)\coma  \cdots\coma  c_n(x))$ for a.e. $x\in X$ (see \cite{HJ13}). By the inductive hypothesis there exists a measurable field $U_2(\cdot):X\rightarrow \cU(n-1)$ such that $d(U_2(x)^* V_1(x) U_2(x))=(c_2(x)\coma  \cdots\coma  c_n(x))$. Hence, if we set $U(x)=(U_1(x)\oplus I_{n-2})\cdot (1\oplus U_2(x))$ for $x\in X$ then $U(\cdot):X\rightarrow \cU(n)$ has the desired properties.

\pausa
If $k>1$ then $b_1(x)\geq\ldots \geq b_{k-1}(x)\geq b_{k}(x)>\eta(x)\geq b_{k+1}(x)\geq \ldots \geq b_n(x)$. 
Let $D_2(x)\in \M_{n-2}(\C)$ be the diagonal matrix with main diagonal 
$$\beta(x)\igdef  (b_1(x)\coma  \ldots\coma b_{k-1}(x)\coma  b_{k+2}(x)\coma  \ldots\coma  b_n(x))\in\R^{n-2}.$$ 
Notice that in this case
$$\begin{pmatrix} U_1(x)& 0 \\
0 &I_{n-2} 
\end{pmatrix}^*
\begin{pmatrix} D_1(x)& 0 \\
0 &D_2(x) 
\end{pmatrix}
\begin{pmatrix} U_1(x) &0 \\
0 &I_{n-2}
\end{pmatrix}=\begin{pmatrix}c_1(x) &W(x)^*\\
W(x) &V_2(x)\end{pmatrix}$$
where $W(x)^*=(\overline{w(x)}\coma  0\coma  \ldots\coma  0)\in M_{1,(n-1)}(\C)$, 
$w(\cdot):X\rightarrow \C$ is a measurable function and $V_2(x)\in M_{n-1}(\C)$ is the diagonal matrix with main diagonal 
$$
\gamma(x)\igdef  (\eta(x)\coma b_1(x)\coma \ldots\coma b_{k-1}(x)\coma b_{k+2}(x)\coma \ldots\coma b_n(x)) \peso{for a.e.} x\in X \ .
$$
It turns out that $(c_2(x)\coma \ldots\coma c_n(x))\prec \gamma(x)$ for a.e. $x\in X$; by the inductive hypothesis there exists a measurable field 
$U_2(\cdot):X\rightarrow \cU(n-1)$ such that $d(U_2(x)^* V_2(x) U_2(x))=(c_2(x)\coma \ldots\coma c_n(x))$ for a.e. $x\in X$. 
Notice that there exists a permutation matrix
$P\in\cU(n)$ such that $P^*(x) D_{b(x)} P=D_1\oplus D_2\,$. Hence, if we set $U(x)=P\cdot (U_1(x)\oplus I_{n-2})\cdot (1\oplus U_2(x))$ for a.e. $x\in X$ then, 
$U(\cdot):X\rightarrow \cU(n)$ has the desired properties.
\end{proof}

\pausa
Next we prove Theorem \ref{teo:mayo equiv}, based on the Schur-Horn theorem for measurable field i.e. Theorem \ref{teo:mayo y el unitario} above. Our approach is an adaptation of some known results in finite frame theory (see \cite{AMRS}).

\pausa
{\bf Theorem \ref{teo:mayo equiv}} \it Let $b:\T^k\rightarrow (\R_+)^d$ and $c:\T^k\rightarrow (\R_+)^n$ be measurable vector fields.
The following statements are equivalent:
\ben
\item For a.e. $x\in \T^k$ we have that $c(x)\prec b(x)$.
\item There exist measurable vector fields $u_j: \T^k\to \C^d$ for $j\in\In$ such that $\|u_j(x)\|=1$ for a.e. $x\in \T^k$ 
and  $j\in \I_n\,$, and such that  
$$
D_{b(x)}=\sum_{j\in \I_n} c_j(x)\,\ u_j(x) \otimes u_j(x) \ ,  \peso{for a.e.} \ x\in \T^k\ .
$$
\een
\rm 
\begin{proof}
First notice that the implication $2.\implies 1.$ follows from well known results in finite frame theory (see \cite{AMRS}) in each point $x\in \T^k$. Hence, we show $1.\implies 2.$ We assume, without loss of generality, that the entries of the vectors $b(x)$ and $c(x)$ are arranged in non-increasing order. We now consider the following two cases:

\pausa
{\bf Case 1:} assume that $n<d$. We let $\tilde c:\T^k\rightarrow \C^d$ be given by $\tilde c(x)=(c(x)\coma 0_{d-n})$ 
for $x\in \T^k$. Then, $\tilde c(x)\prec b(x)$ for $x\in\T^k$ and therefore, by Theorem 
\ref{teo:mayo y el unitario} there exists a measurable field $U(\cdot):\T^k\rightarrow \cU(d)$
such that 
\beq \label{eq aplic SH11}
d(U(x)^* D_{b(x)} \, U(x))=(c_1(x)\coma \ldots\coma c_n(x)\coma 0_{d-n})\peso{for a.e.} x\in\T^k\,.
\eeq Let $v_1(x)\coma \ldots\coma v_d(x)\in\C^d$ denote the columns of $C(x)=D_{b(x)}^{1/2}\,U(x)$,  for $x\in\T^k$. 
Then, Eq. \eqref{eq aplic SH11} implies that: 
$$ \|v_j(x)\|^2= c_j(x) \peso{for} j\in\I_n \ , \quad v_j=0 \peso{for} n+1\leq j\leq d $$
$$ \py 
D_{b(x)}= C(x)\, C(x)^*=\sum_{j\in\I_n} v_j(x)\otimes v_j(x) \peso{for a.e.} x\in\T^k \,.$$
Thus, the vectors $u_j(x)$ are obtained from $v_j(x)$ by normalization, for a.e. $x\in\T^k$ and $j\in\I_n\,$.

\pausa
{\bf Case 2:} assume that $n\geq d$. We let $\tilde b:\T^k\rightarrow \C^n$ be given by 
$\tilde b(x)=(b(x)\coma 0_{n-d})$ for $x\in \T^k$. Then, $c(x)\prec \tilde b(x)$ for $x\in\T^k$ and therefore, by Theorem 
\ref{teo:mayo y el unitario} there exists a measurable field $U(\cdot):\T^k\rightarrow \cU(n)$
such that 
\beq \label{eq aplic SH1}
d(U(x)^* D_{\tilde b(x)} \, U(x))=(c_1(x)\coma \ldots\coma c_n(x))\peso{for a.e.} x\in\T^k\ .
\eeq 
Let $\tilde v_1(x)\coma \ldots\coma \tilde v_n(x)\in\C^n$ denote the columns of $C(x)=D_{\tilde b(x)}^{1/2}U(x)$,  for $x\in\T^k$. 
As before, Eq. \eqref{eq aplic SH1} implies that
$$ 
\|\tilde v_j(x)\|^2= c_j(x) \peso{for} j\in\I_n  \py
D_{\tilde b(x)}= \sum_{j\in\I_n} \tilde v_j(x)\otimes \tilde v_j(x) \peso{for a.e.} x\in\T^k \ .
$$
If we let $\tilde v_j(x)=(v_{i,j}(x))_{i\in\In}$ then, the second identity above implies that $\tilde v_{i,j}(x)=0$ for a.e. $x\in\T^k$ and every $d+1\leq i\leq n$.
If we let $v_j(x)=(v_{i,j}(x))_{i\in\I_d}$ for a.e. $x\in\T^k$ and $j\in\In\,$, we get that 
$$ \|v_j(x)\|^2= c_j(x) \peso{for} j\in\I_n  \py
D_{b(x)}= \sum_{j\in\I_n} v_j(x)\otimes v_j(x) \peso{for a.e.} x\in\T^k \,.$$
Thus, the vectors $u_j(x)$ are obtained from $v_j(x)$ by normalization, for a.e. $x\in\T^k$ and $j\in\I_n\,$.
\end{proof}

\subsection{The reduced finite-dimensional model: proof of Theorem \ref{teo estruc prob reducido unificado}}\label{subsec reduced}

In this section we present the proof of Theorem \ref{teo estruc prob reducido unificado}, divided into two parts (namely, Propositions 
\ref{teo estruc prob reducido} and \ref{era facilongo nomas} below).

\begin{pro}\label{teo estruc prob reducido}
Let $ m,\, n\in\N$, $\alpha\in (\R_{>0}^n)^\downarrow$, $p=(p_i)_{i\in\I_{m}}\in \R_{>0}^m$ and 
$\delta=(d_i)_{i\in\I_{m}}\in\N^m$ be such that $1\leq d_1< \ldots< d_m$.
If $\varphi\in\convf$ then
there exists $\Psi^{\rm op}=[\psi_i^{\rm op}]_{i\in\I_{m}}\in \Lambda^{\rm op}_{\alpha,\,p}(\delta)$ such that 
$$  \sum_{i\in\I_{m}} {p_i}\,\tr(\varphi(\psi_i^{\rm op}) )  
\leq \sum_{i\in\I_{m}} {p_i}\,\tr(\varphi(\psi_i) )
 \peso{for every} \Psi=[\psi_i]_{i\in\I_{m}}\in \Lambda^{\rm op}_{\alpha,\,p}(\delta)\,.$$
Moreover, if $\varphi\in\convfs$ then such $\Psi^{\rm op}$ is unique.
\end{pro}
\begin{proof}
Let us consider the set 
$$
\Lambda_{\alpha,\,p}(\delta)\igdef \bigcup_{B\in W_{\alpha,\,p}} M(B) \inc \prod_{i\in\I_{m}} (\R_+^{d_i})^\downarrow \ ,
$$ 
where 
$$
M(B)\igdef\{ [\lambda_i]_{i\in\I_{m}}\in \prod_{i\in\I_{m}} (\R_+^{d_i})^\downarrow:\ R_i(B)\prec \lambda_i\ , \ i\in\I_{m}\} \ .
$$
Notice that by construction $\Lambda_{\alpha,\,p}^{\rm op}(\delta)\inc \Lambda_{\alpha,\,p}(\delta)$.

\pausa
We claim that $\Lambda_{\alpha,\,p}(\delta)$ is a convex set. 
Indeed, let 
$[\lambda_i]_{i\in\I_{m}}\in M(B_1)$,  
$[\mu_i]_{i\in\I_{m}}\in M(B_2)$ for 
$B_1 $, $B_2\in W_{\alpha,\,p}$  and 
$t\in 
[0,1]$. 
Take the matrix 
$B =  t\, B_1 + (1-t)\,B_2\in W_{\alpha,\,p}\,$ (since $ W_{\alpha,\,p}$ is a convex set). Then  
$$ 
[\,\gamma_i\,]_{i\in\I_{m}} = [\,t\,\lambda_i+(1-t)\, \mu_i\,]_{i\in\I_{m}}\in M(B) \inc \Lambda_{\alpha,\,p}(\delta)
\ :
$$
on the one hand, $\ga_i\in(\R_+^{d_i})^\downarrow$, $i\in\I_m$; on the other hand, by Lidskii's additive inequality (see \cite{Bhat}) we have that, for each  $i\in\I_{m}\,$
$$ 
R_i(B)= t\,R_i(B_1)+ (1-t)\,R_i(B_2)\prec t\, R_i(B_1)^\downarrow + (1-t)\, R_i(B_2)^\downarrow 
\in (\R_+)\da 
\ .
$$ 
On the other hand, by the hypothesis  (and the definition of majorization) one deduces that 
$$
R_i(B_1)\da\prec \la_i \py R_i(B_2)\da\prec \mu_i \implies 
R_i(B) \prec t\, \la_i + (1-t)\, \mu_i= \gamma_i 
$$
for every $i\in  \I_{m}\,$. 
This proves the claim, so $\Lambda_{\alpha,\,p}(\delta)$ is a convex set. Moreover, by the compactness of $W_{\alpha,\,p}$ and by the conditions defining $M(B)$ for $B\in W_{\alpha,\,p}\,$, it follows that $\Lambda_{\alpha,\,p}(\delta)$ is a compact set. 
Let 
$$
\varphi_p:\Lambda_{\alpha,\,p}(\delta)\rightarrow \R_+ \peso{given by} 
\varphi_p(\Psi)\igdef\sum_{i\in\I_{m}} {p_i}\,\tr \,\varphi(\psi_i) \ ,
$$
for $\Psi=[\psi_i]_{i\in\I_{m}} 
\in \Lambda_{\alpha,\,p}(\delta)\,$. 
It is easy to see that $\varphi_p$ is a convex function, which is strictly convex whenever $\varphi\in\convfs$. Using this last fact it follows that there exists 
$ \Psi_0\in  \Lambda_{\alpha,\,p}(\delta)$ that satisfies 
$$ \varphi_p(\Psi_0)\leq  \varphi_p(\Psi) \peso{for every} \Psi\in \Lambda_{\alpha,\,p}(\delta)\ ,
$$ 
and such $ \Psi_0$ is unique whenever $\varphi\in\convfs$. Notice that by construction
there exists some  
$B\in W_{\alpha,\,p}$  such that $\Psi_0=[\psi_i^0]_{i\in\I_{m}} \in M(B)$. 
Then, by item \ref{item2} of Notation \ref{muchas nots},  
$$
R_i(B)\prec \psi_i^0 \implies L_{d_i}(R_i(B))\prec \psi_i^0 
\implies \tr \, \varphi(L_{d_i}(R_i(B)))\leq 
\tr \,\varphi(\psi_i^0) \peso{for} i\in\I_{m}\ .
$$
Hence,  the sequence $B_\delta$ defined in Eq. \eqref{Bdelta} using this matrix $B$ satisfies that 
$\varphi_p(B_\delta)\le \varphi_p(\Psi_0)$. So 
we define $\Psi^{\rm op}\igdef B_\delta\in \Lambda_{\alpha,\,p}^{\rm op}(\delta)\subset \Lambda_{\alpha,\,p}(\delta)$, that has the desired properties. 
Finally, the previous remarks show that $\Psi_0= \Psi^{\rm op}\in \Lambda_{\alpha,\,p}^{\rm op}(\delta)$ whenever $\varphi\in\convfs$. 
\end{proof}

\begin{pro}\label{era facilongo nomas}
With the notations and terminology of Proposition \ref{teo estruc prob reducido}, assume 
further that $n\geq d_m$ and that $\varphi\in\convfs$ is differentiable in 
$\R_+\,$. Then $$ \Psi^{\rm op}\in \prod_{i\in\I_m} (\R_{>0}^{d_i})^\downarrow\,.$$
\end{pro}
\begin{proof}
Let $\Psi^{\rm op}=[\psi_i^{\rm op}]_{i\in\I_{m}}$ where each vector $\psi_i^{\rm op} \in (\R_{+}^{d_i})\da \,$, 
 and assume that there exists $i_0\in\I_m$ such that $\psi_{i_0}^{\rm op}=(\psi_{i_0,j}^{\rm op})_{j\in
\I_{d_{i_0}}}$ satisfies that $\psi^{\rm op}_{i_0,k}=0$ for some $1\leq k\leq d_{i_0}$; let $1\leq k_0\leq d_{i_0}$ be the smallest such index. 
Let $B\in W_{\alpha,\,p}$ be such that $B_\delta=\Psi^{\rm op}$. 
Recall from Eq. \eqref{eq defi gammacd}  that, if we denote $c_i = c_{d_{i}}(R_{i}(B))$ for every $i \in \I_{m}\,$, then  
$$ 
\psi^{\rm op}_{i_0,j}= L_{d_{i_0}} (R_{i_0}(B))_j =
\max\{R_{i_0}(B)\da_j\coma  c_{i_0}\} 
\peso{for} j\in \I_{d_{i_0}} \ , 
$$
since $n\geq d_{i_0}$ by hypothesis. Hence, in this case $c_{i_0}=
 0$ and $R_{i_0}(B)\da_{k_0}=0$. Let $j_0\in \In $ 
such that $0=R_{i_0}(B)\da_{k_0} = B_{i_0\coma j_0}\,$. 
By construction $\sum_{i\in\I_m}p_i\ B_{i\coma j_0}=\alpha_{j_0}>0$ so that there exists $i_1\in\I_m$ such that $B_{i_1,\,j_0}>0$. Let $\{e_j\}_{i\in\In}$ denote the canonical basis of $\R^n$. 
For every $t\in I=[0,\frac{\beta_{i_1,\,j_0} \ p_{i_1}}{p_{i_0}}]$ 
consider the matrix $B(t)$ defined by its rows as follows:
\bit 
\item  $R_{i_0}(B(t))=R_{i_0}(B)+t\, e_{j_0} $ 
\item $R_{i_1}(B(t))= R_{i_1}(B)- \frac{p_{i_0}\, t}{p_{i_1}}\ e_{j_0}\,$ 
\item  $R_{i}(B(t))=R_{i}(B)$ for $i\in\I_m\setminus\{i_0,\,i_1\}$.
\eit
It is straightforward to check that $B(t)\in W_{\alpha,\,p}$ for $t\in I$ and that $B(0)=B$.
Set $\Psi(t)=[\psi_i(t)]_{i\in\I_m}=B(t)_\delta\in \Lambda_{\alpha,\, p}^{\rm op}(\delta)$ for $t\in I$ and notice that $\Psi(0)=\Psi^{\rm op}$. We now consider two cases: 

\pausa {\bf Case 1:\ }\ $B_{i_1,\,j_0}> c_{i_1}$ (recall that 
$\psi^{\rm op}_{i_1\coma j}= L_{d_{i_1}} (R_{i_1}(B))_j =
\max\{R_{i_1}(B)\da_j\coma  c_{i_1}\} $). Therefore $B_{i_1,\,j_0}
=R_{i_1}(B)\da_{k}$ for some $1\leq k\leq d_{i_1}$ and we let $1\leq k_1\leq d_{i_1}$ be the largest such $k$. It is straightforward to check that in this case there exists $\varepsilon >0$ such that 
$$\psi_{i_0}(t)=\psi^{\rm op}_{i_0}+t\, e_{k_0} \py \psi_{i_1}(t)=\psi^{\rm op}_{i_1}-\frac{p_{i_0}}{p_{i_1}}\, t\, e_{k_1} \peso{for} t\in [0,\epsilon]\,.$$
Therefore, for $t\in [0,\epsilon]$ we have that 
$$
f(t)=\varphi_p(\Psi(t))- \varphi_p(\Psi^{\rm op})=p_{i_0}\ (\varphi(t)-\varphi(0))+p_{i_1}\ (\varphi(B_{i_1,\, j_0} - 
\frac{p_{i_0}}{p_{i_1}} t ) - \varphi(B_{i_1,\, j_0}))\  .
$$
Hence $f(0)=0$ and by hypothesis $f(t)\geq 0$ for $t\in[0,\epsilon]$. On the other hand, 
$$
f'(0)= p_{i_0}\ (\varphi'(0) - \varphi'(B_{i_1,\, j_0}) )<0
$$ 
since by the hypothesis $\varphi'$ is strictly increasing and $B_{i_1,\, j_0}>0$. This condition contradicts the previous facts about $f$. From this we see that the vectors in $\Psi^{\rm op}$ have no zero entries. 

\pausa 
{\bf Case 2:\ }\  $B_{i_1,\,j_0}\le c_{i_1}$. Hence, in this case $0<c_{i_1}$ and there exists $0\leq r\leq d_{i_1}-1$ such that 
$$
\psi^{\rm op}_{i_1}=(R_{i_1}(B)\da_1\coma\ldots \coma 
R_{i_1}(B)\da_r \coma c_{i_1}\coma \ldots\coma c_{i_1})
$$ 
so that there exists $\varepsilon>0$ such that for $t\in[0,\varepsilon]$ we have that
$$
\psi_{i_1}(t)=(R_{i_1}(B)\da_1\coma\ldots \coma 
R_{i_1}(B)\da_r \coma c_{i_1}\coma \ldots\coma c_{i_1})-\frac{p_{i_0}\ t}{(d-r)\ p_{i_1}}\sum_{j=r+1}^{d_1} e_j\,.
 $$
Therefore, for $t\in [0,\epsilon]$ we have that 
$$f(t)=\varphi_p(\Psi(t))- \varphi_p(\Psi^{\rm op})=
p_{i_0}\ (\varphi(t)-\varphi(0))+p_{i_1} \ (d-r)\ (\varphi(c_{i_1}- \frac{p_{i_0} \, t}{(d-r)\,p_{i_1} }  ) - \varphi(c_{i_1}))\, .$$
As before, $f(0)=0$ and $f(t)\geq 0$ for $t\in[0,\epsilon]$; a simple computation shows that in this case we also have that $f'(0)<0$, which contradicts the previous facts; thus, the vectors in $\Psi^{\rm op}$ have no zero entries.
\end{proof}

\end{document}